\documentclass[12pt,leqno]{amsart}
\pdfoutput=1
\usepackage{amsmath, amssymb, amsthm}
\usepackage{rsfso}
\usepackage{mathdots, blkarray}
\usepackage[svgnames]{xcolor}
\usepackage{tikz}
\usetikzlibrary{arrows,cd}
\tikzset{>=latex}
\usepackage{mathtools}
\usepackage[alwaysadjust]{paralist}
\usepackage[alphabetic]{amsrefs}
\usepackage[T1]{fontenc}
\usepackage{mathptmx}
\usepackage{microtype}
\usepackage[colorlinks = true,
            linkcolor  = DarkBlue,
            urlcolor   = DarkRed,
            citecolor  = DarkGreen]{hyperref}
\hypersetup{
  pdftitle={Smooth Hilbert Schemes},
  pdfauthor={Roy Skjelnes and Gregory G. Smith}
  }
\usepackage[centering, includeheadfoot, hmargin=1.0in, tmargin=1.0in, 
  bmargin=1in, headheight=6pt]{geometry}
\newtheorem{lemma}{Lemma}[section]
\newtheorem{theorem}[lemma]{Theorem}
\newtheorem{maintheorem}{Theorem}
\newtheorem{corollary}[lemma]{Corollary}
\newtheorem{proposition}[lemma]{Proposition}
\theoremstyle{definition}
\newtheorem{definition}[lemma]{Definition}
\newtheorem{remark}[lemma]{Remark}
\newtheorem{examplex}[lemma]{Example}
\newenvironment{example}
  {\pushQED{\qed}\examplex}
  {\popQED\endexamplex}
\newtheorem{question}[lemma]{Question}

\newcommand{\PP}{\ensuremath{\mathbb{P}}} 
\newcommand{\Gr}[2]{\ensuremath{\operatorname{Gr}\bigl(#1,\PP(#2)\!\bigr)}}
\newcommand{\Fl}[2]{\ensuremath{\operatorname{Flag}\bigl(#1,\PP(#2)\!\bigr)}}
\newcommand{\Res}[2]{\ensuremath{\operatorname{Flag}\bigl(#1,#2,\PP(E)\!\bigr)}}
\DeclareMathOperator{\Hilb}{Hilb}
\DeclareMathOperator{\Hom}{Hom}
\DeclareMathOperator{\Proj}{Proj}
\DeclareMathOperator{\Spec}{Spec}
\DeclareMathOperator{\Sym}{Sym}

\begin{document}

\vspace*{-3em}

\title[Smooth Hilbert Schemes]{Smooth Hilbert Schemes\\[-3pt]
  \mbox{\footnotesize their classification and geometry}
  \vspace{-0.5em}}

\author[R.~Skjelnes]{Roy Skjelnes}
\address{Roy Skjelnes: Department of Mathematics, Royal Institute of Technology
  (KTH), Stockholm, 100~44, Sweden;%
  {\normalfont \texttt{skjelnes@math.kth.se}}}

\author[G.G.~Smith]{Gregory G.{} Smith}
\address{Gregory G.{} Smith: Department of Mathematics and Statistics, Queen's
  University, Kingston, Ontario, K7L 3N6, Canada;%
  {\normalfont \texttt{ggsmith@mast.queensu.ca}}} 


\vspace*{-0.5em}

\begin{abstract}
  Closed subschemes in projective space with a fixed Hilbert polynomial are
  parametrized by a Hilbert scheme.  We classify the smooth ones.  We identify
  numerical conditions on a polynomial that completely determine when the
  Hilbert scheme is smooth.  We also reinterpret these smooth Hilbert schemes
  as generalized partial flag varieties and describe the subschemes being
  parametrized.
\end{abstract}

\maketitle

\vspace*{-1.5em}
\section*{Overview}

\noindent
Hilbert schemes are crucial for compactifying families of subschemes and
constructing moduli spaces.  Among these parameter spaces, Hilbert schemes of
points on a projective surface are exceptional.  Being smooth, they have a
wider range of applications including deep results in algebraic geometry,
combinatorics, and representation theory; see \cites{Bea, Hai, Groj, Nak}.  In
contrast, little is known about geometric properties of other Hilbert schemes.
Even the geometry of $\Hilb^{p}(\PP^{m})$, the Hilbert scheme parametrizing
closed subschemes in projective $m$-space $\PP^{m}$ with Hilbert polynomial
$p$, is poorly understood when $m \geqslant 3$.  Although
Hartshorne~\cite{Har} shows that each $\Hilb^{p}(\PP^{m})$ is path-connected,
celebrated insights into these Hilbert schemes typically highlight
pathologies. For example, Mumford~\cite{Mum} exhibits an irreducible component
in $\Hilb^{14t-23}(\PP^3)$ that is generically non-reduced,
Ellia--Hirschowitz--Mezzetti~\cite{EHM} show that the number of irreducible
components in $\Hilb^{dt+c}(\PP^3)$ is not bounded by a polynomial in
$\mathbb{Q}[c,d]$, and Vakil~\cite{Vak} proves that every singularity type
appears in some $\Hilb^{p}(\PP^{\,4})$.  As a counterpoint, this article
classifies the smooth Hilbert schemes $\Hilb^{p}(\PP^{m})$ and describes their
geometry.


Our primary theorem uses integer partitions to characterize smooth Hilbert
schemes.  An integer partition $\lambda$ is an $r$-tuple
$\lambda \coloneqq (\lambda_1, \lambda_2, \dotsc, \lambda_r)$ of integers
satisfying
$\lambda_1 \geqslant \lambda_2 \geqslant \dotsb \geqslant \lambda_r \geqslant
1$.

\begin{maintheorem}
  \label{t:main}
  For any positive integer $m$ and any polynomial $p$ in $\mathbb{Q}[t]$, the
  Hilbert scheme $\Hilb^{p}( \PP^{m})$ is a smooth irreducible variety if and
  only if there exists an integer partition
  $\lambda \coloneqq (\lambda_1, \lambda_2, \dotsc, \lambda_r)$ such that
  \[
    p(t) = \sum_{i=1}^{r} \binom{t + \lambda_i - i}{\lambda_i -1} 
  \]
  and one of the following seven conditions holds:
  \begin{compactenum}[\upshape (1)]
  \item $m = 2 \geqslant \lambda_1$,
  \item $m \geqslant \lambda_1$ and $\lambda_r \geqslant 2$, 
  \item $\lambda = (1)$ or $\lambda = (m^{r-2}, \lambda_{r-1}, 1) =
    (\smash{\overbrace{m, m, \dotsc, m}^{\text{\tiny\upshape $(r \!-\! 2)$-times}}},
    \lambda_{r-1}, 1)$ 
    where $r \geqslant 2$ and $m \geqslant \lambda_{r-1} \geqslant 1$,
  \item $\lambda = (m^{r-s-3}, \lambda_{r-s-2}^{s+2}, 1)$ where
    $r - 3 \geqslant s \geqslant 0$ and
    $m - 1 \geqslant \lambda_{r-s-2} \geqslant 3$,
  \item $\lambda = (m^{r-s-5}, 2^{s+4}, 1)$ where
    $r-5 \geqslant s \geqslant 0$,
  \item $\lambda = (m^{r-3}, 1^3)$ where $r \geqslant 3$,
  \item $\lambda = (m+1)$ or $r = 0$.
  \end{compactenum}
\end{maintheorem}

These combinatorial conditions encode the underlying geometry. For instance,
the integer partition $\lambda = (\lambda_1)$ corresponds to the Grassmannian
of $(\lambda_1-1)$-dimensional planes in $\PP^{m}$ and
$\lambda = (m^r) = (m, m, \dotsc, m)$ corresponds to the Hilbert scheme
parametrizing hypersurfaces of degree $r$ in $\PP^{m}$; both well-known
families are covered by Condition~2 and the case $\lambda = (1)$ in
Condition~3.  More generally, every point on a Hilbert scheme satisfying
Condition~3 where $\lambda = (m^d, \ell, 1)$ corresponds to the
scheme-theoretic union of a hypersurface of degree $d$, a linear subspace of
dimension $\ell - 1$, and a point.  Similarly, general points on a Hilbert
scheme satisfying Condition~5 with $\lambda = (m^{d}, 2^{c}, 1)$ correspond to
the union of a hypersurface of degree $d$, a plane curve of degree $c$, and a
point, whereas those satisfying Condition~4 with
$\lambda = (m^{d}, \ell^{c}, 1)$ correspond to the union of a hypersurface of
degree $d$, a hypersurface of degree $c$ contained in an $\ell$-dimensional
linear subspace, and a point.  The minor discrepancies in Conditions~4 and~5,
arising from the integer partitions $\lambda = (2^2,1)$ and
$\lambda = (2^3,1)$, are required because the Hilbert schemes with points
corresponding to two skew lines and a twisted cubic curve are singular; see
Example~\ref{e:gaps}.  For Condition~6, a general point on the Hilbert scheme
with $\lambda = (m^{d}, 1^3)$ corresponds to the union of a hypersurface of
degree $d$ and $3$ reduced points.  For completeness, observe that the unique
point on a Hilbert scheme satisfying Condition~7 corresponds to either $\PP^m$
or the empty scheme.

The list of conditions in Theorem~\ref{t:main} is new and answers Lin's
question~\cite{Lin}.  However, the challenge lies in proving that this list is
exhaustive.  Understanding the geometry of Condition~2 is, unexpectedly, the
key to overcoming this challenge.  Our geometric interpretation in this
condition relies on expanding the traditional notion of a residual scheme.  To
be more precise, consider a hypersurface $D$ in $\mathbb{P}^{m}$.  The
residual scheme of a closed immersion $D \subseteq X$ in $\PP^{m}$ is the
unique closed subscheme $Y \subset X$ such that their defining ideal sheaves
on $\PP^{m}$ satisfy $\mathcal{I}_X = \mathcal{I}_Y \cdot \mathcal{I}_D$.
Geometrically, the scheme $X$ is the union of $Y$ and $D$.  Building on this
concept, a closed immersion $Y \subset X$ in $\PP^{m}$ is a \emph{residual
  inclusion} if there exists a linear subspace $\Lambda$ in $\PP^{m}$
containing $X$ and a hypersurface $D$ in $\Lambda$ such that $Y$ is the
residual scheme of $D \subseteq X$ in $\Lambda$.  We define a \emph{residual
  flag} in $\PP^{m}$ to be a chain
\[
  \varnothing = X_{e+1} \subset X_{e} \subset \dotsb \subset X_1
\]
such that, for all $1 \leqslant i \leqslant e$, the closed immersion
$X_{i+1} \subset X_{i}$ is a residual inclusion; see
Definition~\ref{d:rflag}.  Unlike other flags, the scheme $X_i$ routinely
fails to be equidimensional.  Informally, a residual flag extends a partial
flag like a multiset extends a set: the degree of each hypersurface in a
residual flag is analogous to the multiplicity of each element in a multiset.
Proposition~\ref{p:rep} demonstrates that the parameter spaces
representing residual flags are projective bundles over partial flag
varieties.

Beyond the classification in Theorem~\ref{t:main}, our second major
contribution proves that a general point on the smooth Hilbert schemes
satisfying Conditions~2--7 corresponds to either a residual flag or the union
of a residual flag and a point.  For all integers $m$ greater than $2$, this
describes the closed subschemes parametrized by a smooth $\Hilb^p(\PP^{m})$.
In the first case, we deduce that these smooth Hilbert schemes are projective
bundles over partial flag varieties; see Theorem~\ref{t:resH}.  In
particular, every point on a Hilbert scheme satisfying Conditions~2--3
corresponds to a residual flag.  In the second case, the smooth Hilbert
schemes are birational to the product of $\PP^{m}$ and a projective bundle
over a partial flag variety; see Proposition~\ref{p:birat} and
Example~\ref{e:pts}.  In other words, we realize the smooth Hilbert schemes
$\Hilb^p(\PP^{m})$ as suitable generalizations of partial flag varieties.

The success in classifying these smooth Hilbert schemes suggests new questions
that may be tractable. What conditions on the integer partition $\lambda$
imply that $\Hilb^{p}(\PP^{m})$ is irreducible?  How does one extend this
result to Quot schemes or nested Hilbert schemes?  What is the right analogue
if $\PP^{m}$ is replaced with a smooth toric variety, a complete intersection,
or a Grassmannian?

\subsection*{Strategy of proof}

We analyze $\Hilb^p(\PP^{m})$ via the induced action of the general linear
group.  A point in this Hilbert scheme is Borel-fixed if the stabilizer of the
corresponding closed subscheme contains all lower triangular matrices.  Every
nonempty $\Hilb^{p}(\PP^{m})$ has distinguished Borel-fixed point called the
lexicographic point. This point shows that rewriting the polynomial $p$ in
terms of the integer partition $\lambda$ is equivalent to $\Hilb^{p}(\PP^{m})$
being nonempty; see \cite{Mac} or \cite{Har}*{Corollary~5.7}.
Reeves--Stillman~\cite{RS}*{Theorems~1.4 and~4.1} establish that the
lexicographic point is always smooth and determine the dimension of the unique
irreducible component containing it.  Hence, proving smoothness reduces, in
principle, to computing the dimension of the tangent space at the other
Borel-fixed points.  Unfortunately, the rapid growth in the number of these
points and the complexity of individual points overwhelm a brute-force attack.

To circumvent these complications, we identify a new family of subschemes in
$\mathbb{P}^m$ that correspond to singular points on $\Hilb^{p}(\PP^{m})$.
Residual flags are used to describe these points and to prove that they are
singular.  By exploiting the geometry of residual flags, our analysis reduces
to the Hilbert schemes with an integer partition
$\lambda = \bigl( (m-1)^{d}, \ell^{c}, 1)$ where $m-2 \geqslant \ell \geqslant 1$;
see Proposition~\ref{p:sing}.  In this situation, we construct an explicit
monomial ideal and exhibit a number of linearly independent deformations; see
Lemma~\ref{l:sing}.  Since this number exceeds the dimension of the
lexicographic component, this ideal corresponds to a singular point on
$\Hilb^{p}(\PP^{m})$.

In hindsight, the complete classification of smooth Hilbert schemes is obtain
from just a few families of Borel-fixed points.  Only three singular families,
in addition to our new family, are required; see
Examples~\ref{e:gaps}$\,$--$\,$\ref{e:4pts}.  For smoothness, only one family
other than the lexicographic points is needed; see Theorem~\ref{t:resH} and
Proposition~\ref{p:birat}.  This reduction may be the most surprising
development.  The relevant Borel-fixed points also, fortuitously, avoid
technicalities arising in positive characteristic, thereby producing uniform
results over the integers.

Amusingly, all seven conditions in Theorem~\ref{t:main} correspond to Hilbert
schemes that are already known to be smooth.  Fogarty's
article~\cite{Fogarty}*{Theorems~1.4 and~2.4} shows that Conditions~1 and~6
guarantee smoothness.  Serving as our initial inspiration, Staal's
thesis~\cite{St1}*{Theorem~1.1} establishes that Conditions~2--3 correspond to
smooth Hilbert schemes.  Likewise, Ramkumar's
preprint~\cite{Ramkumar}*{Theorem~A} proves that Conditions~4--5 are
associated to smooth Hilbert schemes.  Under Condition~7, the Hilbert scheme
is just one point.  Despite recognizing each condition, the consolidated list
does not already appear in the literature and is not obviously complete.

\subsection*{Computational experience}

Although independent of our proofs, calculations in \emph{Macaulay2}~\cite{M2}
were indispensable in the discovery of our results.  Recoding the Hilbert
polynomial as an integer partition gives a novel method of sampling nonempty
Hilbert schemes.  Using \emph{Macaulay2}, we made a systematic search of the
Borel-fixed points, for all $3 \leqslant m \leqslant 7$, exposing
Conditions~2--6.  We learn, a posteriori, that Conditions~2--5 imply that
Hilbert schemes have at most two Borel-fixed points.  Our computational
experiments suggest that the number of parts in the integer partition
$\lambda$ equal to $1$ governs the size of the intersection graph for the
irreducible components in the Hilbert scheme.

\section{Residual flags}

In this section, we introduce the notion of a residual flag and show that the
scheme parametrizing these objects is smooth and projective. Throughout, we
work over a locally noetherian base scheme $S$ and $E$ denotes a coherent
$\mathcal{O}_S$-module.  We write
$\PP(E) \coloneqq \Proj \bigl( \Sym(E) \!\bigr)$ for the projectivization of
the graded symmetric algebra $\Sym(E)$.

\subsection*{Grassmannians}
For any $S$-scheme $T$, let $E_T$ denote the pull-back of $E$ to $T$.  A
surjection of coherent $T$-modules $E_T \to F$ gives a closed immersion
\[
  \PP(F) \subseteq \PP(E_T \!) = \PP(E) \times_{\!S} T\, .
\]
If $F$ is locally free of constant rank $n + 1$, then the subscheme $\PP(F)$
is a $n$-plane in $\PP(E_T \!)$.  The set of $T$-valued points of the
functor $\Gr{n}{E}$ is the set of the $n$-planes in $\PP(E_T \!)$.  The
$S$-scheme representing this functor is projective; see
\cite{EGA}*{Proposition~9.8.4}.  When $E$ is locally free of constant rank
$m + 1$, the map $\Gr{n}{E} \to S$ is smooth of relative dimension
$(n + 1)(m - n)$.

\subsection*{Flag varieties}
Consider an $e$-tuple $n \coloneqq (n_1, n_2, \dotsc, n_e)$ of integers such
that $n_1 > n_2 > \dotsb > n_{e} \geqslant 0$.  For any $S$-scheme $T$, a flag
of type $n$ in $\PP(E_T \!)$ is a chain of closed immersions
\[
  \PP(F_e) \subset \PP(F_{e-1}) \subset \dotsb \subset \PP(F_1)
\]
where each $\PP(F_i)$ is a $n_i$-plane in $\PP(E_T\!)$ for all
$1 \leqslant i \leqslant e$.  The set of $T$-valued points of the functor
$\Fl{n}{E}$ is the set of flags of type $n$ in $\PP(E_T \!)$.  The
$S$-scheme representing this functor is projective; see
\cite{EGA}*{Proposition~9.9.3}.  A flag is a succession of Grassmannians.
Hence, when $E$ is locally free of constant rank $n_0 + 1$, it follows that
the map $\Fl{n}{E} \to S $ is smooth of relative dimension
$\sum_{i=1}^{e} (n_i +1)(n_{i-1} - n_{i})$.

\subsection*{Relative divisors}
A closed subscheme $D \subset \PP(E)$ is a relative effective Cartier divisor
if it is flat over $S$ and its ideal sheaf is invertible.  The divisor $D$ has
degree $d$ if, for each geometric point $\Spec(k) \to S$, the fibre
$D \times_{\!S} \Spec(k)$ is a hypersurface of degree $d$ in
$\PP^{m} \cong \PP(E) \times_{\!S} \Spec(k)$; compare with
\cite{KM}*{Corollary~1.1.5.2}.  Using the dual sheaf
$E^* \coloneqq \mathcal{H}\!\mathit{om}(E, \mathcal{O}_{S})$, we may
parametrize these divisors in $\PP(E)$; see \cite{Fogarty}*{Proposition~1.2}
and \cite{Kollar}*{Exercise~1.4.1.4}.

\begin{lemma}
  \label{l:Cart}
  Assume that $E$ is a locally free sheaf and let $E^*$ be its dual.  For all
  nonnegative integers $d$, the $S$-scheme $\PP \bigl( \Sym^d(E^*) \!\bigr)$
  represents the functor of relative effective Cartier divisors in $\PP(E)$
  having degree $d$.
\end{lemma}

\begin{proof}
  Let $T$ be an $S$-scheme and set
  $F_T \coloneqq \smash{\bigl( \Sym^{d}(E^*) \! \bigr)}_T = \Sym^d \bigl(
  (E^*)_T \bigr)$.  Given a line bundle $L$ on $T$ and a surjection
  $F_T \to L$, we see that $L^*$ is an invertible subsheaf of $F_T^*$.  The
  ideal sheaf generated by $L^*$ in $\Sym(E^*)_T$ is invertible and determines
  a hypersurface of degree $d$ fiberwise in $\PP(E_T\!)$.  Thus, it defines a
  relative effective Cartier divisor in $\PP(E_T\!)$ of degree $d$.
\end{proof}

\subsection*{Residual scheme}
Consider the closed immersion $D \subseteq X$ in $\PP(E)$ where $D$ is a
relative effective Cartier divisor.  Let $\mathcal{I}_D$ and $\mathcal{I}_X$
denote the ideal sheaves of the closed subschemes $D$ and $X$ in $\PP(E)$. The
residual scheme to $D$ in $X$ is the closed subscheme $Y$ in $\PP(E)$ defined
by the colon ideal sheaf
$\mathcal{I}_{Y} \coloneqq (\mathcal{I}_X \mathbin{:} \mathcal{I}_D) =
\mathcal{I}_X \cdot {\mathcal{I}_{D}}^{\!\!\smash{-1}}$. It follows that
$\mathcal{I}_X = \mathcal{I}_Y \cdot \mathcal{I}_D$ and $X$ is the union of
the subschemes $D$ and $Y$; see \cite{Fulton}*{Definition~9.2.1} and
\cite{Fogarty}*{pp.~512--513}.

\begin{definition}
  For any positive integer $d$, a closed immersion $Y \subset X$ in $\PP(E)$
  is a \emph{$d$-residual inclusion} if there exists a relative effective
  Cartier divisor $D$ in $\PP(E)$ of degree $d$ such that the closed subscheme
  $Y$ is the residual scheme to $D$ in $X$ with respect to $\PP(E)$.
\end{definition}

\begin{lemma}
  \label{l:flat}
  Let $d$ be a positive integer and let $Y \subset X$ in $\PP(E)$ be a
  $d$-residual inclusion. The map $Y \to S$ is flat if and only if the map
  $X \to S$ is flat.
\end{lemma}

\begin{proof}
  The existence of a relative effective Cartier divisor $D$ in $\PP(E)$ such
  that $Y$ is the residual scheme to $D$ in $X$ yields the short exact
  sequence 
  \[
    \begin{tikzcd}
      0 \arrow[r] & \mathcal{O}_Y(-D) \arrow[r] & \mathcal{O}_X \arrow[r] &
      \mathcal{O}_D \arrow[r] & 0 \, ,
    \end{tikzcd}
  \]
  where multiplication by a local equation for $D$ defines the injective map.
  By hypothesis, the sheaf $\mathcal{O}_D$ is flat over $S$.  We deduce that
  $\mathcal{O}_Y$ and $\mathcal{O}_Y(-D)$ are flat over $S$ if and only if
  $\mathcal{O}_X$ is flat over $S$.
\end{proof}

\begin{definition}
  \label{d:rflag} 
  Let
  $(n, d) \coloneqq (n_1, d_1), (n_2, d_2), \dotsc, (n_{e}, d_{e})$ be
  a sequence of pairs of positive integers such that
  $n_1 > n_2 > \dotsb > n_e > 0$.  For any $S$-scheme $T$, a
  \emph{residual flag} of type $(n, d)$ in $\PP(E_T\!)$ is a chain of closed
  immersions
  \[
    \varnothing = X_{e+1} \subset X_{e} \subset X_{e-1} \subset
    \dotsb \subset X_1 
  \]
  in $\PP(E_T\!)$ such that, for all $1 \leqslant i \leqslant e$, the
  following properties are satisfied:
  \begin{compactenum}[(i)]
  \item the scheme $X_i$ is flat over $T$,
  \item the scheme $X_i$ is contained in some $n_i$-plane
    $\PP(F_i) \subseteq \PP(E_T \!)$, and
  \item the closed immersion $X_{i+1} \subset X_{i}$ is a $d_i$-residual
    inclusion in $\PP(F_{i})$.
  \end{compactenum}
\end{definition}

\begin{remark}
  \label{r:bot}
  For any residual flag, the third property for $i = e$ asserts that closed
  immersion $\varnothing = X_{e+1} \subset X_{e}$ is a $d_{e}$-residual
  inclusion.  In other words, the closed subscheme $X_{e}$ is a relative
  effective Cartier divisor of degree $d_{e}$ in some $n_{e}$-plane
  $\PP(F_{e}) \subseteq \PP(E_T \!)$.
\end{remark}

\begin{remark}
  \label{r:d=1}
  Residual flags generalize flags of linear subspaces. To be more explicit,
  assume $E$ is locally free of constant rank $m + 1$ and let $T$ be an
  $S$-scheme.  Given a flag
  $\Lambda_{e} \subset \Lambda_{e-1} \subset \dotsb \subset \Lambda_1$ of type
  $n$ in $\PP(E_T \!)$ where $m > n_1$, there exists a flag
  $\PP(F_{e}) \subset \PP(F_{e-1}) \subset \dotsb \subset \PP(F_1)$ of type
  $(n_1+1, n_2+1, \dotsc, n_{e}+1)$ in $\PP(E_T \!)$ such that the $n_i$-plane
  $\Lambda_i$ is a hyperplane in $\PP(F_i)$ for all
  $1 \leqslant i \leqslant e$.  Setting $X_{e+1} \coloneqq \varnothing$, we
  define $X_i \coloneqq \Lambda_i \cup X_{i+1}$ by a descending induction.  It
  follows that $X_{i+1} \subset X_i$ is a $1$-residual inclusion for all
  $1 \leqslant i \leqslant e$. Thus, the chain of closed immersions
  $\varnothing \subset X_{e} \subset X_{e-1} \subset \dotsb \subset X_{1}$ is
  a residual flag of type $(n_1+1, 1), (n_2+2, 1), \dotsc, (n_e +1, 1)$ in
  $\PP(E_T \!)$.
\end{remark}

\begin{example}
  We illustrate how to recursively construct the defining ideal for the closed
  subschemes in a residual flag.  Let $(n, d) \coloneqq (3,2), (2,4)$ and
  set $S = \Spec(k)$ where $k$ is a field.  The residual flags of type
  $(n, d)$ in $\PP^3 \coloneqq \Proj(k[x_0, x_1, x_2, x_3]) $ are nested
  pairs $X_2 \subset X_1$ in $\PP^3$ such that $X_2$ is a planar curve of
  degree $4$ and $X_2$ is a $2$-residual scheme in $X_1$.  The defining ideal
  of $X_2$ has the form 
  \[
    I_{X_2} \coloneqq \langle f_1, f_2 \rangle
  \]
  where $f_1$ is a homogeneous polynomial in $k[x_0,x_1,x_2,x_3]$ of degree
  $1$ and $f_2$ is a homogeneous polynomial in $k[x_0,x_1,x_2,x_3]$ of degree
  $d_2 = 4$ that is not divisible by $f_1$.  The $2$-plane $\PP(F_2)$
  containing $X_2$ is given by the vanishing of the linear form $f_1$. The
  defining ideal of the closed subscheme $X_1$ in $\PP^3$ has the form
  \[
    I_{X_1} \coloneqq g \cdot I_{X_2} = \langle g f_1, g f_2 \rangle
  \]
  where $g \in k[x_0,x_1,x_2,x_3]$ is a homogeneous polynomial of degree
  $d_1 = 2$.  Geometrically,
  the scheme $X_1$ is the union of the quadratic hypersurface defined by the
  vanishing of $g$ and the planar quartic curve $X_2$.  For the special
  configuration in which $g = f_1^2$, the defining ideal of the closed
  subscheme $X_1$ is
  $I_{X_1} = \langle f_1^3, f_1^2 f_2^{} \rangle = \langle f_1^2 \rangle \cap
  \langle f_1^3, f_2 \rangle$.
\end{example}

\subsection*{Functor of residual flags}
The pullback of a residual scheme is again a residual scheme; see
\cite{Fogarty}*{Lemma~1.3}.  It follows that residual flags define a
contravariant functor from the category of $S$-schemes to the category of
sets.  Let $(n, d) \coloneqq (n_1, d_1), (n_2, d_2), \dotsc, (n_{e}, d_{e})$
be a sequence of pairs of positive integers such that
$n_1 > n_2 > \dotsb > n_{e}$ and let $E$ be coherent $\mathcal{O}_{S}$-module.
For any $S$-scheme $T$, the set of $T$-valued points of the functor
$\Res{n}{d}$ is defined to be the set of residual flags of type $(n, d)$ in
$\PP(E_T\!)$.

\begin{lemma}
  \label{l:1pair}
  Assume that $(n_1,d_1)$ is a pair of positive integers and $E$ is a
  coherent sheaf on $S$. Let $F$ be the universal quotient sheaf on the
  Grassmannian $\Gr{n_1}{E}$ and let $F^{*}$ denote its dual.
  \begin{compactenum}[\upshape (i)]
  \item When $d_1 > 1$, the $S$-scheme $\PP \bigl( \Sym^{d_1}(F^{*}) \!\bigr)$
    represents the functor of residual flags of type $(n_1, d_1)$ in
    $\PP(E)$.  The structure map of this $S$-scheme is the composition of the
    canonical maps $\PP \bigl( \Sym^{d_1}(F^{*}) \!\bigr) \to \Gr{n_1}{E}$
    and $\Gr{n_1}{E} \to S$.
  \item When $d_1 = 1$, the Grassmannian $\Gr{n_1-1}{E}$ represents the
    functor of residual flags of type $(n_1, 1)$ in $\PP(E)$.
  \end{compactenum}
\end{lemma}

\begin{proof}
  Let $T$ be an $S$-scheme. Remark~\ref{r:bot} shows that residual flags of
  type $(n_1, 1)$ in $\PP(E_T\!)$ are $(n-1)$-planes, so part~(ii) follows.
  immediately.  Assume that $d_1 > 1$.  A $T$\nobreakdash-valued point of
  $\PP\bigl( \Sym^{d_1}(F^{*}) \bigr)$ consists of a line bundle $L_T$ on $T$
  and a surjection $\Sym^{d_1} \bigl( \! (F_T \!)^* \!\bigr) \to L_T$ together
  with a $T$\nobreakdash-valued point of $\Gr{n_1}{E}$.  Since
  $\smash{\bigl( \Sym^{d_1}(F^{*}) \!\bigr)}_T = \Sym^{d_1} \bigl( \! (F_T
  \!)^* \bigr)$, Lemma~\ref{l:Cart} demonstrates that $L_T$ corresponds to a
  relative effective Cartier divisors of degree $d_1$ in the $n_1$-plane
  $\PP(F_T\!)$.  The $T$\nobreakdash-valued point of $\Gr{n_1}{E}$ corresponds
  to the $n_1$-plane $\PP(F_T\!) \subseteq \PP(E_T\!)$.  Thus, the
  $S$\nobreakdash-scheme $\PP \bigl( \Sym^{d_1}(F^{*}) \!\bigr)$ represents
  the residual flags of type $(n_1, d_1)$ in $\PP(E)$.
\end{proof}

\begin{remark}
  When the base scheme $S$ is the spectrum of a field $k$ and $d_1 > 1$, the
  parameter space for the residual flags of type $(n_1, d_1)$ in
  $\PP^{m} \coloneqq \Proj(k[x_0, x_1, \dotsc, x_{m}])$ is the variety of
  degree $d_1$ hypersurfaces in $n_1$-planes in $\PP^{m}$; see
  \cite{Fulton}*{Example~14.7.12}.
\end{remark}

\subsection*{Latent planes}
The definition of a residual flag
$\varnothing \subset X_{e} \subset X_{e -1} \subset \dotsb \subset X_{1}$ of
type $(n, d)$ in $\PP(E_T\!)$ includes the existence of flag of linear
subspaces.  For all $1 \leqslant i \leqslant e$, the scheme $X_i$ lies in some
$n_i$-plane $\PP(F_i) \subseteq \PP(E_T\!)$.  When $d_{e} > 1$, we refer to
the set $\{ \PP(F_i) \mid 1 \leqslant i \leqslant e \}$ as the latent planes
of the residual flag.  In the special case $d_{e} = 1$, the scheme $X_{e}$ is
itself a $(n_{e} - 1)$-plane and the latent planes are
$\{ X_{e} \} \cup \{ \PP(F_i) \mid 1 \leqslant i \leqslant e - 1 \}$.

\begin{lemma}
  \label{l:late}
  Let $(n, d)$ be the type of a residual flag.
  \begin{compactenum}[\upshape (i)]
  \item When $d_{e} > 1$, there exists a morphism
    $\Res{n}{d} \to \Fl{n}{E}$ sending a residual flag to its flag
    of latent planes.
  \item When $d_{e} = 1$, there exists a morphism
    $\Res{n}{d} \to \Fl{n^{\circ}}{E}$ sending a residual flag to its flag
    of latent planes where
    $n^{\circ} \coloneqq (n_1, n_2, \dotsc, n_{e-1}, n_{e}-1)$.
  \end{compactenum}
\end{lemma}

\begin{proof}
  We need to show that the latent planes are unique and form a flag.  Let $T$
  be an $S$-scheme and let
  $\varnothing \subset X_{e} \subset X_{e -1} \subset \dotsb \subset X_1$ be a
  residual flag of type $(n, d)$ in $\PP(E_T\!)$.  Suppose that, for some
  $1 \leqslant i \leqslant e$, the scheme $X_i$ is contained in two distinct
  $n_i$-planes $\PP(F_i)$ and $\PP({F_i}')$.  The closed immersion
  $X_{i+1} \subset X_i$ is a $d_i$-residual inclusion in $\PP(F_i)$, so there
  exists a relative effective Cartier divisor $D_i$ in $\PP(F_i)$ such that
  $X_i$ is the union of $D_i$ and $X_{i+1}$.  Since the codimension of $X_i$
  in $\PP(F_i)$ equals $1$, we deduce that
  $X_i = D_i = \PP(F_i) \cap \PP({F_i}')$.  It follows that either the
  $n_i$-plane containing $X_i$ is unique or $i = 1$, $D_1 = X_1$, and
  $d_1 = 1$.  Thus, each scheme $X_i$ is contained in a unique plane having
  the dimension of its corresponding latent plane, so both assertions follow.
\end{proof}

\subsection*{Representability} 
The pivotal result in this section shows that the functor of residual
flags is representable.  Moreover, it realizes this parameter space as a
generalization of a partial flag variety.

\begin{proposition}
  \label{p:rep}
  Assume that
  $(n, d) \coloneqq (n_1, d_1), (n_2, d_2), \dotsc, (n_e, d_e)$ is the
  type of a residual flag and $E$ is a coherent sheaf on $S$. For all
  $1 \leqslant i \leqslant e$, let $F_i^{*}$ denote the dual of the universal
  quotient sheaf on the Grassmannian $\Gr{n_i}{E}$.
  \begin{compactenum}[\upshape (i)]
  \item When $d_{e} > 1$, we have the Cartesian square%
    {
      \[
        \begin{tikzcd}[row sep = 12pt, column sep = 20pt]
          \Res{n}{d} \arrow[r] \arrow[d]
          & \PP \bigl( \Sym^{d_1}(F_1^{*}) \! \bigr) \times_{\!S}
          \PP \bigl( \Sym^{d_2}(F_2^{*}) \!\bigr) \times_{\!S} \dotsb
          \times_{\!S}
          \PP \bigl( \Sym^{d_{e}}(F_{e}^{*}) \! \bigr) \arrow[d] \\
          \Fl{n}{E} \arrow[r] &
          \Gr{n_1}{E} \times_{\!S}
          \Gr{n_2}{E} \times_{\!S} \dotsb \times_{\!S}
          \Gr{n_{e}}{E} \, .
        \end{tikzcd}
      \]}%
  \item When $d_{e} = 1$, setting
    $n^{\circ} \coloneqq (n_1, n_2, \dotsc, n_{e-1}, n_{e}-1)$ gives
    the Cartesian square%
    {
      \[
        \begin{tikzcd}[row sep = 12pt, column sep = 10pt]
          \!\!\!\! \Res{n}{d} \!\!\! \arrow[r] \arrow[d] & \!\!
          \PP \bigl( \Sym^{d_1}(F_1^{*}) \! \bigr) \!\times_{\!S}
          \PP \bigl( \Sym^{d_2}(F_2^{*}) \! \bigr) \!\times_{\!S}
          \dotsb \!\times_{\!S}
          \PP \bigl( \Sym^{d_{e-1}}(F_{e-1}^{*}) \! \bigr)
          \!\times_{\!S} \Gr{n_{e}-1}{E} \arrow[d] \hspace{-12pt} \\
          \!\!\!\! \Fl{n^{\circ}}{E} \arrow[r] &
          \Gr{n_1}{E} \times_{\!S}
          \Gr{n_2}{E} \times_{\!S} \dotsb \times_{\!S}
          \Gr{n_{e - 1}}{E} \times_{\!S}
          \Gr{n_{e} - 1}{E} \, . \hspace{-12pt}
        \end{tikzcd}
      \]}%
  \end{compactenum}
  In both cases, the functor $\Res{n}{d}$ is represented by a projective
  $S$-scheme.
\end{proposition}

\begin{proof}
  The bottom arrows in the diagrams are closed immersions; see
  \cite{EGA}*{Proposition~9.9.3}.  To prove the projectivity of
  $\Res{n}{d}$, it is enough to show that these squares are Cartesian.

  Assume that $d_{e} > 1$. Let $T$ be an $S$-scheme and let
  $\varnothing \subset X_{e} \subset X_{e-1} \subset \dotsb \subset X_1$ be a
  residual flag of type $(n, d)$ in $\PP(E_T\!)$.  For all
  $1 \leqslant i \leqslant e$, there exists a relative effective Cartier
  divisor $D_i$ in $\PP(F_i)$ such that $X_i = D_i \cup X_{i+1}$.
  Lemma~\ref{l:1pair} implies that the product of projective bundles
  \[
    P \coloneqq
    \PP \bigl( \Sym^{d_1}(F_1^{*}) \! \bigr) \times_{\!S}
    \PP \bigl( \Sym^{d_2}(F_2^{*}) \! \bigr) \times_{\!S} \dotsb \times_{\!S}
    \PP \bigl( \Sym^{d_{e}}(F_{e}^{*}) \! \bigr)
  \]
  over
  $G \coloneqq \Gr{n_1}{E} \times_{\!S} \Gr{n_2}{E} \times_{\!S} \dotsb
  \times_{\!S} \Gr{n_{e}}{E}$ parametrizes $e$-tuples of relative effective
  Cartier divisors of degree $d_i$ contained in some $n_i$-plane in $\PP(E)$
  for all $1 \leqslant i \leqslant e$.  Hence, we have a morphism from
  $\Res{n}{d}$ to $P$ sending the residual flag
  $\varnothing \subset X_{e} \subset X_{e-1} \subset \dotsb \subset X_1$ to
  the $e$-tuple
  $\bigl( D_1 \!\subset \PP(F_1), D_2 \!\subset \PP(F_2), \dotsc, D_{e}
  \!\subset \PP(F_{e}) \! \bigr)$.  Combined with the morphism in
  Lemma~\ref{l:late}, we obtain a morphism from $\Res{n}{d}$ to the
  fibre product $\Fl{n}{E} \times_{\!G} P$.

  We want to exhibit the inverse of this morphism. A $T$-valued point of this
  fibre product is a flag
  $\PP(F_{e}) \subset \PP(F_{e-1}) \subset \dotsb \subset \PP(F_1)$ of type
  $n$ in $\PP(E_T\!)$ together with an $e$-tuple of relative effective Cartier
  divisors $D_i$ in $\PP(F_i)$ for all $1 \leqslant i \leqslant e$.  Setting
  $X_{e+1} \coloneqq \varnothing$, we define $X_i \coloneqq D_i \cup X_{i+1}$
  for all $1 \leqslant i \leqslant e$ by descending induction.  By
  construction, we have a chain
  $\varnothing \subset X_{e} \subset X_{e-1} \subset \dotsb \subset X_1$ of
  closed immersions in $\PP(E_T\!)$.  Since Lemma~\ref{l:flat} shows that, for
  all $1 \leqslant i \leqslant e$, the scheme $X_i$ is flat over $T$, this
  chain is a residual flag of type $(n, d)$ in $\PP(E_T\!)$.  As the pullback
  of a residual scheme is again a residual scheme, this construction is also
  functorial.  We conclude that the scheme $\Res{n}{d}$ and the fibre product
  $\Fl{n}{E} \times_{\!G} P$ are isomorphic.

  The proof for the case $d_{e} = 1$ is very similar.  Since the scheme
  $X_{e}$ is a $(n_{e}-1)$-plane in $\PP(E_T\!)$, we have a natural morphism
  from $\Res{n}{d}$ to the product
  \[
    \PP \bigl( \Sym^{d_1}(F_1^{*}) \! \bigr) \times_{\!S} \PP \bigl(
    \Sym^{d_2}(F_2^{*}) \! \bigr) \times_{\!S} \dotsb \times_{\!S} \PP \bigl(
    \Sym^{d_{e-1}}(F_{e-1}^{*}) \! \bigr) \times_{\!S} \Gr{n_{e}}{E}
  \]
  sending the residual flag to the tuple
  $\bigl( D_1 \subset \PP(F_1), D_2 \subset \PP(F_2), \dotsc, D_{e-1} \subset
  \PP(F_{e-1}), X_{e} \subset \PP(E_T\!) \! \bigr)$.  Using
  Lemma~\ref{l:late}, we obtain a morphism from $\Res{n}{d}$ to the
  appropriate fibre product.  As above, one exhibits the inverse morphism by
  defining $X_i \coloneqq D_i \cup X_{i+1}$ for all
  $1 \leqslant i \leqslant e-1$.
\end{proof}

\begin{corollary}
  \label{c:dim}
  Let $(n, d) \coloneqq (n_1, d_1), (n_2, d_2), \dotsc, (n_e, d_e)$ be
  the type of a residual flag and let $E$ be a locally free sheaf on $S$ of
  constant rank $n_0 + 1$ where $n_0 \geqslant n_1$.
  \begin{compactenum}[\upshape (i)]
  \item When $d_{e} > 1$, the structure map $\Res{n}{d} \to S$ is smooth
    of relative dimension
    \[
      \sum_{i=1}^{e} \left[
        \binom{n_i + d_i}{d_i} - 1 +
        (\! n_i + 1)(\! n_{i-1} - n_i)
      \right] \, .
    \]
  \item When $d_{e} = 1$, the structure map $\Res{n}{d} \to S$ is smooth
    of relative dimension
    \[
      - (\! n_{e-1} - n_{e}) + \sum_{i=1}^{e}
      \left[
        \binom{n_i + d_i}{d_i} - 1 
        + (\! n_i + 1)(\! n_{i-1} - n_i)
      \right] \, .
    \]
  \end{compactenum}
\end{corollary}

\begin{proof}
  For all $1 \leqslant i \leqslant e$, the sheaf $\Sym^{d_i}(F_i^{*})$ on the
  Grassmannian $\Gr{n_i}{E}$ is locally free of constant rank
  $\smash{\binom{n_i+d_i}{d_i}}$.  Using the relative dimensions of flag
  varieties and Grassmannians, Proposition~\ref{p:rep} shows that the map
  $\Res{n}{d} \to S$ is smooth of the claimed relative dimensions.
\end{proof}

\begin{question}
  \label{q:chow}
  When $E$ is locally free, the fibre product interpretation leads to a
  presentation for the Chow ring of $\Res{n}{d}$. Specifically, one
  combines the formula for the Chow ring of a projective bundle with the
  description of the Chow ring of a partial flag variety; see
  \cite{Fulton}*{Examples~8.3.4 and 14.7.16}.  Moreover, the cycle map on
  $\Res{n}{d}$ from its Chow ring to its integral cohomology ring is an
  isomorphism; see \cite{Fulton}*{Example~19.1.11}.  Which aspects of Schubert
  calculus on partial flag varieties extend to the parameter space of residual
  flags?
\end{question}

\begin{question}
  For line bundles on $\PP(E)$, the higher-direct images under the structure
  map to the base scheme $S$ are well-understood.  Similarly, the
  Borel--Weil--Bott theorem describes the higher-direct images of line bundles
  on a flag variety under the structure map to the base scheme.  What is the
  common refinement for line bundles on the parameter space $\Res{n}{d}$?
\end{question}

\section{Hilbert polynomials and residual flags}

Using the geometry of residual flags, we explain the combinatorial formula for
the Hilbert polynomials. We show that the type of a residual flag encodes the
Hilbert polynomial of its largest subscheme.  We also prove that every
lexicographic ideal (other than the zero ideal and unit ideal) determines a
residual flag.  Let $R \coloneqq k[x_0, x_1, \dotsc, x_{m}]$ denote the
standard graded polynomial ring over a field $k$ and set
$\PP^{m} \coloneqq \Proj(R)$.

\subsection*{Integer partitions}
We repackage the type of a residual flag as a single integer partition.  Given
a sequence $(n_1, d_1), (n_2, d_2), \dotsc, (n_{e}, d_{e})$ of pairs of
positive integers such that $n_1 > n_2 > \dotsb > n_e > 0$, the integer
partition $\lambda \coloneqq (\lambda_1, \lambda_2, \dotsc, \lambda_r)$
satisfies
$\lambda_1 \geqslant \lambda_2 \geqslant \dotsb \geqslant \lambda_r \geqslant
1$ and
\[
  \lambda = (
  \underbrace{n_1, n_1, \dotsc, n_1}_{\text{$d_1$-times}},
  \underbrace{n_2, n_2, \dotsc, n_2}_{\text{$d_2$-times}},
  \dotsc,
  \underbrace{n_e, n_e, \dotsc, n_e}_{\text{$d_e$-times}}
  )
  \, .
\]
The length of $\lambda$ is $r \coloneqq d_1 + d_2 + \dotsb + d_e$.  It can be
convenient to use a notation for integer partitions that indicates the number
of times each integer occurs as a part; see \cite{Mac15}*{Subsection~1.1}.
The expression $\lambda = ( \dotsc, i^{a_i}, \dotsc, 2^{a_2}, 1^{a_1} )$ means
that, for all positive integers $i$, exactly $a_i$ of the parts in $\lambda$
are equal to $i$.  For instance, we have
$\lambda = (\! n_1^{\,\smash{d_1}}, n_2^{\,\smash{d_2}}, \dotsc,
n_{e}^{\,\smash{d_{e}}})$.

To describe Hilbert polynomials, we treat a binomial coefficient with a
variable in its numerator as a polynomial. Specifically, for all integers $c$,
we set
\[
  \binom{t}{c} \coloneqq
  \begin{cases}
    \tfrac{1}{c!} (t)(t-1) \dotsb (t-c+1) & \text{if $c \geqslant 0$} \\
    0 & \text{if $c < 0$.  }
  \end{cases}
\]
The polynomial $\binom{t}{c}$ has rational coefficients and degree $c$.
  
\begin{lemma}
  \label{l:poly}
  Let $(n, d)$ be the type of a residual flag and let
  $\lambda \coloneqq (\lambda_1, \lambda_2, \dotsc, \lambda_r)$ be its
  associated integer partition.  For any residual flag
  $\varnothing \subset X_{e} \subset X_{e-1} \subset \dotsb \subset X_1$ of
  type $(n,d)$ in $\PP^{m}$, the Hilbert polynomial of the closed scheme
  $X_1$ in $\PP^{m}$ is
  \[
    p(t) = \sum_{i=1}^{r} \binom{t + \lambda_i - i}{\lambda_i - 1} \, .
  \]
\end{lemma}

\begin{proof}
  We proceed by induction on $e$.  The closed immersion $X_{2} \subset X_{1}$
  is a $d_{1}$-residual inclusion in some $n_1$-plane contained in
  $\PP^{m}$.  Hence, there exists a relative effective Cartier divisor
  $D_{1}$ in this $n_1$-plane and a short exact sequence of sheaves
  \begin{equation}
    \label{f:ses}
    \begin{tikzcd}
      0 \arrow[r]
      & \mathcal{O}_{X_{2}}(-D_{1}) \arrow[r]
      & \mathcal{O}_{X_1} \arrow[r] &
      \mathcal{O}_{D_1} \arrow[r] & 0 \, .
    \end{tikzcd}
  \end{equation}
  When $e = 1$, we have $X_{2} = \varnothing$. The closed scheme $X_{1}$ is
  the divisor $D_{1}$, so its Hilbert polynomial is
  \begin{align*}
    p(t)
    &= \binom{t + n_{1}}{n_{1}}
      - \binom{t + n_{1} - d_{1}}{n_{1}} 
    = \sum_{i=1}^{d_{1}} \left[
      \binom{t + n_{1}-i+1}{n_{1}}
      - \binom{t + n_{1} - i}{n_{1}}
      \right]
      = \sum_{i=1}^{d_{1}} \binom{t + n_1 - i}{n_1 - 1} \, . 
  \end{align*}
  The integer partition associated to the residual flag
  $\varnothing \!\subset X_{1}$ is $\lambda = (n_{1}^{\smash{d_{1}}})$, so
  the base case holds.

  Suppose that $e > 1$. The residual flag
  $\varnothing \subset X_{e} \subset X_{e-1} \subset \dotsb \subset X_{2}$ has
  $e-1$ subschemes and its integer partition is
  $(n_{2}^{d_{2}}, n_{3}^{d_{3}}, \dotsc,
  n_{e\vphantom{2}}^{d_{e\vphantom{2}}})$.  The induction hypothesis implies
  that the Hilbert polynomial of the closed scheme $X_{2}$ is
  $\sum_{i=d_1+1}^{r} \binom{t+\lambda_i - i + d_{1}}{\lambda_i -1}$.  From
  the short exact sequence \eqref{f:ses}, we deduce that the Hilbert
  polynomial of the closed scheme $X_{1}$ is
  \begin{align*}
    p(t)
    &= \sum_{i = d_1+1}^{r}
      \binom{t + \lambda_i - i + d_1}{\lambda_i -1}
      + \sum_{i=1}^{d_{1}} \binom{t + n_j - i}{n_j - 1}
      = \sum_{i=1}^{r}
      \binom{t + \lambda_i - i}{\lambda_i -1} \, . \qedhere
  \end{align*}
\end{proof}

\begin{remark}
  Let $X$ be a closed subscheme in $\PP^{m}$ having Hilbert polynomial
  $\smash{p(t) = \sum_{i=1}^{r} \binom{t + \lambda_i - i}{\lambda_i -1}}$. The
  dimension of $X$ is $\lambda_1 -1$ and the degree of $X$ is the number $d_1$
  of parts in $\lambda$ equal to $\lambda_1$.  One verifies that the
  arithmetic genus of $X$ is
  \[
    (-1)^{\lambda_1 -1} \sum_{i=2}^{r} \binom{\lambda_i -i}{\lambda_i -1} =
    \sum_{i=2}^{r} (-1)^{\lambda_1 - \lambda_i} \binom{i-2}{\lambda_i -1} \, .
  \]
  Remarkably, the Gotzmann Regularity Theorem shows that the length $r$ of
  $\lambda$ is an upper bound on the Castelnuovo-Mumford regularity of the
  saturated ideal defining the closed subscheme $X$ in $\PP^{m}$; see
  \cite{Got}*{Lemma~2.9} or \cite{BH}*{Theorem~4.3.2}
\end{remark}

\subsection*{Lexicographic ideals}
We identify a special residual flag by recognizing the geometric properties of
a distinguished monomial ideal.  The lexicographic order on the monomials in
$R = k[x_0,x_1, \dotsc, x_{m}]$ is defined by declaring
$x_0^{\,\smash{b_0}} \, x_1^{\,\smash{b_1}} \, \dotsb \, x_{m}^{\,\smash{b_m}}
> x_0^{\,\smash{c_0}} \, x_1^{\,\smash{c_1}} \, \dotsb \,
x_{m}^{\,\smash{c_m}}$ whenever the first nonzero entry in the integer
sequence $(b_0 - c_0, b_1 - c_1, \dotsc, b_{m} - c_{m})$ is positive.  A
lexicographic ideal $I$ is a monomial ideal in $R$ such that, for all integers
$j$, the homogeneous component $I_{\!j}$ is the $k$-vector space spanned by
the largest $\dim_{k} I_{\!j}$ monomials in lexicographic order.

As the cornerstone of our approach, we recount a variant of the Macaulay
characterization~\cite{Mac} of the Hilbert functions for homogeneous ideals in
polynomial ring $R$.  Recall that the Hilbert function
$h_{R/I} \colon \mathbb{Z} \to \mathbb{N}$ of a homogeneous ideal $I$ in $R$
is defined, for all integers $j$, by $h_{R/I}(j) \coloneqq \dim_{k} (R/I)_j$.

\begin{lemma}
  \label{l:nonempty}
  Let $p$ be a numerical polynomial having degree less than $m$. The
  following statements are equivalent.
  \begin{compactenum}[\upshape (a)]
  \item There exists a closed subscheme $X$ in $\PP^{m}$ whose Hilbert
    polynomial is $p$.
  \item There exists an integer partition
    $\lambda \coloneqq (\lambda_1, \lambda_2, \dotsc, \lambda_r)$ such that
    $p(t) = \smash[t]{\sum\nolimits_{i=1}^{r}
        \binom{t + \lambda_i - i}{\lambda_i - 1}}$.
  \item There exists an integer partition
    $\lambda \coloneqq (\lambda_1, \lambda_2, \dotsc, \lambda_r)$ such that
    $p(t) = \smash[b]{\sum\nolimits_{i=1}^{r}
      \binom{t + \lambda_i - i}{\lambda_i - 1}}$
    and a \newline lexicographic ideal $L(\lambda)$ in $R$ such that
    $h_{R/L(\lambda)}(j) = \smash{\sum\nolimits_{i=1}^{r}} \binom{j+\lambda_i
      -i}{j-i+1}$ for all integers $j$.
  \end{compactenum}
\end{lemma}

\begin{proof}[Outline of Proof]
  We sketch the details because a proof may be derived from other accounts of
  Macaulay's work via elementary identities for binomial coefficients; see
  \cite{GKP}*{Table~174}.
  \begin{compactitem}[$\bullet$]
  \item[\textit{(a $\Rightarrow$ b)}] Let $\ell$ be a fixed sufficiently large
    integer.  One uses the $\ell$-th Macaulay representation for the integer
    $p(\ell)$ to obtain an integer partition
    $\lambda \coloneqq (\lambda_1, \lambda_2, \dotsc, \lambda_r)$ such that
    $\smash{p(t) = \sum_{i=1}^{r} \binom{t + \lambda_i -i}{\lambda_i -1}}$;
    see \cite{BH}*{Lemma~4.2.6 and Corollary~4.2.14} or
    \cite{St1}*{Proposition~2.3}.
  \item[\textit{(b $\Rightarrow$ c)}] One verifies that the function
    $h \colon \mathbb{Z} \to \mathbb{N}$ defined, for all integers $j$, by
    $\smash{h(j) \coloneqq \sum_{i=1}^r \binom{j+\lambda_i-i}{j-i+1}}$
    satisfies the Macaulay inequality
    $\smash{\bigl( h(j) \!\bigr)}^{\!\langle j \rangle} \geqslant h(j)$; see
    \cite{BH}*{Theorem~4.2.10}.
  \item[\textit{(c $\Rightarrow$ a)}] One takes $X$ to be the closed subscheme
    of $\PP^n$ defined by the lexicographic ideal $L(\lambda)$. \qedhere
  \end{compactitem}
\end{proof}

\begin{remark}
  By using the conjugate integer partition, Corollary~5.7 in \cite{Har} gives
  an alternative condition equivalent to Lemma~\ref{l:nonempty}~(b); see
  \cite{St1}*{Lemma~2.4}.
\end{remark}

We specify monomial generators and a primary decomposition of the
lexicographic ideal $L(\lambda)$; these generators are also listed in
\cite{RS}.

\begin{proposition}
  \label{p:irr}
  Let $\lambda \coloneqq (\lambda_1, \lambda_2, \dotsc, \lambda_r)$ be an
  integer partition and let $a_{\!j}$ be the number of parts in $\lambda$
  equal to $j$, for all positive integers $j$.  When $n \geqslant \lambda_1$,
  the corresponding lexicographic ideal is
  \[
    L(\lambda) = \bigl\langle
    x_{0}^{\,\smash{a_{m}}+1},
    x_0^{\, \smash{a_{m}}} \, x_1^{\,\smash{a_{m-1}+1}}, \dotsc,
    x_0^{\,\smash{a_{m}}} \, x_1^{\,\smash{a_{m-1}}} \, \dotsb \,
    x_{m-3}^{\,\smash{a_{3}}} \, x_{m-2}^{\,\smash{a_{2}+1}},
    x_0^{\,\smash{a_{m}}} \, x_1^{\,\smash{a_{m-1}}} \, \dotsb \,
    x_{m-2}^{\,\smash{a_{2}}} \, x_{m-1}^{\,\smash{a_{1}}}
    \bigr\rangle \, . 
  \]
  Moreover, the unique irredundant irreducible decomposition of this monomial
  ideal is
  \[
    L(\lambda) = 
    \bigcap_{\begin{subarray}{c}
        1 \leqslant i \leqslant m \\
        a_i \neq 0
      \end{subarray}}
    \bigl\langle
    x_{0}^{\,\smash{a_{m}} + 1}, x_{1}^{\,\smash{a_{m-1}}+1}, \dotsc,
    x_{m-i-1}^{\,\smash{a_{i+1}}+1}, x_{m-i}^{\,\smash{a_{i}}}
    \bigr\rangle \,.
    \vspace{-5pt}
  \]
\end{proposition}

\begin{proof}
  We first establish the decomposition.  As each intersectand is generated by
  powers of the variables and no two have the same dimension, this
  intersection is the irredundant irreducible decomposition of some monomial
  ideal.  It remains to show that this ideal is $L(\lambda)$.  Since
  $a_{\!j} = 0$ for all $j > \lambda_1$, each irreducible ideal contains
  $\langle x_0, x_1, \dotsc, x_{m - \lambda_1 - 1} \rangle$ and we may
  assume that $\lambda_1 = m$.

  We proceed by induction on the number $e$ of positive entries in
  $(a_1, a_2, \dotsc, a_{m})$.  When $e = 1$, the integer partition is
  $(m^{a_m})$ and the ideal is $\langle x_0^{\,\smash{a_m}} \rangle$.  The
  principal ideal $\langle x_0^{\,\smash{a_m}} \rangle$ is lexicographic.
  Since the monomials $\{1, x_0^{}, \dotsc, x_0^{\,\smash{a_m -1}}\}$ form a
  basis as free $k[x_1,x_2, \dotsc, x_m]$-module for the quotient
  $R/\langle x_0^{\,\smash{a_n}} \rangle$, it follows that
  $h_{R/\langle x_0^{\,\smash{a_m}} \rangle}(j) = \smash{\sum_{i=1}^{a_{m}}
    \binom{j + m - i}{j-i+1}}$.  By Lemma~\ref{l:nonempty}~(c), we deduce that
  $L(m^{a_m}) = \langle x_0^{\,\smash{a_m}} \rangle$ and the base case holds.

  Now, assume that $e > 1$.  Set
  $I \coloneqq L(m^{a_{m}}) = \langle x_0^{\,\smash{a_m}} \rangle$.  The
  induction hypothesis implies that
  \[
    J \coloneqq 
        \bigcap_{\begin{subarray}{c}
        1 \leqslant i \leqslant m - 1 \\
        a_i \neq 0
      \end{subarray}}
    \bigl\langle
    x_{0}^{\,\smash{a_{m}} + 1}, x_{1}^{\,\smash{a_{m - 1}}+1}, \dotsc,
    x_{m-i-1}^{\,\smash{a_{i+1}}+1}, x_{m-i}^{\,\smash{a_{i}}}
    \bigr\rangle
    \vspace{-5pt}
  \]
  is the lexicographic ideal associated to the integer partition
  $\nu \coloneqq (\lambda_{a_m + 1}, \lambda_{a_m +2}, \dotsc,
  \lambda_r)$.  Since the intersection of lexicographic ideals is again
  lexicographic, it is enough to prove that the ideals $I \cap J$ and
  $L(\lambda)$ have the same Hilbert function.  From the short exact sequences
  of graded $R$-modules
  \[
    \begin{tikzcd}[row sep = -1pt, column sep = 15pt]
      0 \arrow[r]
      & \dfrac{R}{I \cap J} \arrow[r]
      & \dfrac{R}{I} \oplus \dfrac{R}{J} \arrow[r]
      & \dfrac{R}{I+J} \arrow[r]
      & 0 \, \phantom{,}
      &[-10pt] \text{and} \phantom{,} \\
      0 \arrow[r]
      & \dfrac{R(-a_{n})}{J} \arrow[r]
      & \dfrac{R}{J} \arrow[r]
      & \dfrac{R}{I+J} \arrow[r]
      & 0 \, ,  
    \end{tikzcd}
  \]
  we see that
  $h_{R/I \cap J}(j) = h_{R/I}(j) + h_{R/J}(j) - h_{R/(I+J)}(j) = h_{R/I}(j) +
  h_{R/J}(j-a_{n})$ for all integer $j$.  The equality $J = L(\nu)$ implies
  that
  $h_{R/J}(j) = \sum_{i = a_{m}+1}^{r} \binom{j + \lambda_i - i + a_m}{j -
    i + a_{m}+1}$, so we deduce that
  \[
    h_{R/I \cap J}(j)
    = \sum_{i = 1}^{a_{m}} \binom{j + \lambda_i - i}{j - i + 1}
    + \sum_{i = a_{m}+1}^{r} \binom{j + \lambda_i - i}{j - i + 1}
    = \sum_{i = 1}^{r} \binom{j + \lambda_i - i}{j - i + 1} \, ,
  \]
  which by Lemma~\ref{l:nonempty}~(c) completes the induction.

  Lastly, we establish that the given set of monomials generate $L(\lambda)$.
  The intersection of monomial ideals is generated by the least common
  multiples of their monomial generators, so we observe 
  \begin{align*}
    x_0^{\,\smash{a_{m}}} \, x_1^{\,\smash{a_{m-1}}} \, \dotsb \,
    x_{m-i-1}^{\,\smash{a_{i+1}}} \, x_{m-i}^{\,\smash{a_{i}+1}}
    &= \operatorname{lcm} \bigl(
      x_{m-i}^{\,\smash{a_{i}+1}}, x_{m-i}^{\,\smash{a_{i}+1}}, \dotsc,
      x_{m-i}^{\,\smash{a_{i}+1}}, x_{m-i}^{\,\smash{a_{i}}},
      x_{m-i-1}^{\,\smash{a_{i+1}}}, \dotsc, x_{0}^{\,\smash{a_{m}}}
      \bigr) 
    \\
    x_0^{\,\smash{a_{m}}} \, x_1^{\,\smash{a_{m-1}}} \, \dotsb \,
    x_{m-2}^{\,\smash{a_{2}}} \, x_{m-1}^{\,\smash{a_{1}}}
    &= \operatorname{lcm} \bigl(
      x_{m-1}^{\,\smash{a_{1}}}, x_{m-2}^{\,\smash{a_{2}}}, \dotsc,
      x_{0}^{\,\smash{a_{m}}}
      \bigr) \, ,
  \end{align*}
  for all $2 \leqslant i \leqslant m$.
  Hence, each of the given monomials is a least common multiple of a generator
  from the irreducible components.  Since each variable $x_i$ appears as a
  minimal generator in an irreducible component with exponent either
  $a_{m-i}$ or $a_{m-1}+1$, we see that the least common multiple of any
  subset of generators for the irreducible components is divisible by at least
  one of the given monomials, so the opposite inclusion also holds.
\end{proof}

We complete our converse to Lemma~\ref{l:poly} by relating lexicographic
ideals to residual flags.  This geometric interpretation for the lexicographic
ideal $L(\lambda)$ appears to be new.

\begin{corollary}
  \label{c:lex}
  Let $(n, d)$ be the type of a residual flag in $\PP^{m}$.  For all
  $1 \leqslant i \leqslant e$, let $X_i$ be the closed subscheme in $\PP^{m}$
  defined by the lexicographic ideal
  $L(n_{i\vphantom{+1}}^{d_{i\vphantom{+1}}}, n_{\,i+1}^{d_{i+1}}, \dotsc,
  n_{e\vphantom{+1}}^{d_{e\vphantom{+1}}})$.  The chain of closed immersions
  $\varnothing \subset X_e \subset X_{e-1} \subset \dotsb \subset X_1$ forms a
  residual flag of type $(n,d)$ in $\PP^{m}$.
\end{corollary}

\begin{proof}
  For all $1 \leqslant i \leqslant e$, the irreducible decomposition in
  Proposition~\ref{p:irr} implies that $X_{i+1} \subset X_{i}$ and the
  monomial generators in Proposition~\ref{p:irr} establish that the closed
  subscheme $X_i$ is contained in the $n_i$-plane in $\PP^{m}$ defined by
  the monomial ideal $\langle x_0, x_1, \dotsc, x_{m - n_i - 1} \rangle$.
  For all positive integers $j$, let $a_{\!j}$ denote the number of parts in
  the integer partition
  $(n_{1}^{\smash{d_{1}}}, n_{2}^{\smash{d_{2}}}, \dotsc,
  n_{e\vphantom{1}}^{\smash{d_{e\vphantom{1}}}})$ equal to $j$.  Restricting to the
  linear subspace
  $\PP^{n_i} \coloneqq \Proj(k[x_{m - n_i}, x_{m - n_i +1}, \dotsc,
  x_{m}])$ in $\PP^{m}$, Proposition~\ref{p:irr} also shows that the
  closed subscheme $X_i$ is defined by the monomial ideal
  \begin{align*}
    I_i &\coloneqq \bigl\langle
    x_{m - n_i}^{\,\smash{a_{n_i}}+1}, \;
    x_{m - n_i}^{\, \smash{a_{n_i}}} \,
    x_{m - n_i-1}^{\,\smash{a_{n_i +1}+1}}, \; \dotsc, \;
    x_{m - n_i}^{\, \smash{a_{n_i}}} \,
    x_{m - n_i-1}^{\,\smash{a_{n_i +1}+1}} \dotsb \,
    x_{m - 3}^{\,\smash{a_{3}}} \, x_{m - 2}^{\,\smash{a_{2}+1}}, \;
    x_{m - n_i}^{\, \smash{a_{n_i}}} \,
    x_{m - n_i-1}^{\,\smash{a_{n_i +1}+1}} \, \dotsb \,
    x_{m - 2}^{\,\smash{a_{2}}} \, x_{m - 1}^{\,\smash{a_{1}}}
    \bigr\rangle \, . 
  \end{align*}
  It follows that $I_i = x_{m - n_i}^{\, \smash{a_{n_i}}} \cdot J$
  where
  \begin{align*}
    J
    &\coloneqq \bigl\langle
      x_{m - n_i}^{}, \; 
      x_{m - n_i-1}^{\,\smash{a_{n_i +1}+1}}, \; \dotsc, \;
      x_{m - n_i-1}^{\,\smash{a_{n_i +1}+1}} \dotsb \,
      x_{m - 3}^{\,\smash{a_{3}}} \, x_{m - 2}^{\,\smash{a_{2}+1}}, \;
      x_{m - n_i-1}^{\,\smash{a_{n_i +1}+1}} \, \dotsb \,
      x_{m - 2}^{\,\smash{a_{2}}} \, x_{m - 1}^{\,\smash{a_{1}}}
      \bigr\rangle \\
    &= 
      \bigl\langle
      x_{n-n_i}^{}, x_{n-n_i-1}^{}, \dotsc, x_{n-n_{i+1}-1}^{}
      \bigr\rangle + I_{i+1} \, . 
  \end{align*}
  Since $a_{n_i} = d_{i}$, the closed immersion $X_{i+1} \subset X_{i}$ is a
  $d_{i}$-residual inclusion in $\PP^{n_i}$.  Therefore, the chain of closed
  immersions
  $\varnothing \subset X_{e} \subset X_{e-1} \subset \dotsb \subset X_1$ is a
  residual flag of type $(n, d)$ in $\PP^{m}$.
\end{proof}

\begin{remark}
  By definition, the lexicographic point in a Hilbert scheme corresponds to
  the closed subscheme in $\PP^{m}$ defined by the lexicographic ideal.
  Theorem~1.4 in \cite{RS} proves that the lexicographic point is smooth, so
  this point lies on a unique irreducible component called the lexicographic
  component.  Theorem~4.1 in \cite{RS} computes the dimension of this
  component.  When
  $(n, d) \coloneqq (n_1, d_1), (n_2, d_2), \dotsc, (n_e, d_e)$ is the type of
  a residual flag, $n_0 \coloneqq m$, and $\lambda$ is its associated integer
  partition, the lexicographic component determined by $L(\lambda)$ in
  $R \coloneqq k[x_0, x_1, \dotsc, x_{m}]$ is%
  {
    \begin{equation}
      \label{f:dim}
      \hspace*{-8pt}
      \begin{cases}
        \displaystyle\sum\limits_{i=1}^{e}  \left[  
          \dbinom{n_i + d_i}{d_i}  \! -   1 +
          (n_i +  1)(n_{i-1} - n_i)
          \right]
        & \hspace*{-10pt} \text{if $n_e > 1$ and $d_{e} > 1$,} \\
        - (n_{e-1} - n_{e})
        +  \displaystyle\sum\limits_{i=1}^{e}   \left[ 
          \dbinom{n_i + d_i}{d_i} \! -   1
          + (n_i +  1)(n_{i-1} - n_i)
          \right]
        & \hspace*{-10pt} \text{if $n_e  > 1$ and $d_{e} = 1$,} \\
        n_0 \, d_{e} +  \displaystyle\sum\limits_{i=1}^{e-1} \left[ 
          \dbinom{n_i + d_i}{d_i} \! -  1 + (n_i +  1)(n_{i-1} - n_i)
          \!\right]
        & \hspace*{-10pt} \text{if $n_e = 1$ and $d_{e-1} > 1$,} \\
        n_0 \, d_{e} - (n_{e-2} - n_{e-1})
        + \displaystyle\sum\limits_{i=1}^{e-1}   \left[ 
          \dbinom{n_i + d_i}{d_i}  -  1 + (n_i +  1)(n_{i-1} - n_i)
          \right] 
        & \hspace*{-10pt} \text{if $n_e  = 1$ and $d_{e-1} = 1$,} \\
        n_0 \, d_{e}
        & \hspace*{-10pt} \text{if $n_{e} = 1$ and $e = 1$.}
      \end{cases}
      \hspace*{-10pt}
    \end{equation}}%
\end{remark}

\begin{question}
  The saturated monomial ideals defining residual flags of type $(n, d)$ in
  $\PP^{m}$ determine the torus-fixed points in the parameter space
  $\Res{n}{d}$.  Following Proposition~\ref{p:irr}, the irredundant
  irreducible decomposition for each such monomial ideal has a combinatorial
  description.  Can one use this perspective to count these monomial ideals
  and, thereby, compute the Euler characteristic of the projective scheme
  $\Res{n}{d}$?
\end{question}

\section{Geometry of smooth Hilbert schemes}

In this section, we identify the Hilbert schemes isomorphic to a parameter
space of residual flags.  Exploiting this identification, we describe the
closed subscheme corresponding to a general point on any smooth Hilbert
scheme.  Thus, we obtain a birational description of all smooth Hilbert
schemes.

\subsection*{Geometry}
Let $E$ be a locally free sheaf on a locally noetherian scheme $S$.  The
projective bundle $\PP(E)$ carries a tautological invertible sheaf that is
relatively ample over $S$. We compute the Hilbert polynomials for closed
subschemes in $\PP(E)$ relative to this tautological bundle. For any numerical
polynomial $p$, the set of $T$-valued points of the functor
$\Hilb^p\bigl(\PP(E) \!\bigr)$ is the set of closed subschemes
$X \subseteq \PP(E_T \!)$ that are flat over the $S$-scheme $T$ and have
Hilbert polynomial $p$. The $S$-scheme representing this functor is
projective; see \cite{Kollar}*{Theorem~1.4}.

\begin{lemma}
  \label{l:loc}
  Let $E$ be locally free sheaf on $S$ of constant rank $m + 1$ and fix a
  polynomial $p$ in $\mathbb{Q}[t]$. The Hilbert scheme
  $\Hilb^{p} \bigl( \PP(E) \! \bigr)$ is smooth over $S$ if and only if the
  fibre $\Hilb^{p}(\PP^{m})$ is nonsingular over every geometric point of $S$.
\end{lemma}

\begin{proof}
  Let $X \subseteq \PP^{m}$ be a closed subscheme in the fibre of $\PP(E)$
  over a geometric point in $S$. When $X$ corresponds to a smooth point on
  $\Hilb^{p}(\PP^{m})$, Theorem~2.10 in \cite{Kollar} proves that the
  structure map $\Hilb^{p} \bigl( \PP(E)\!\bigr) \to S$ is flat at this
  geometric point. Therefore, this structure map is smooth if and only if its
  the fibre is nonsingular over every geometric point.
\end{proof}

\begin{theorem}
  \label{t:resH}
  Let $(n, d)$ be the type of a residual flag and let
  $\lambda \coloneqq (\lambda_1, \lambda_2, \dotsc, \lambda_r)$ be its
  associated integer partition. Set
  $p(t) = \sum_{i=1}^{r} \binom{t + \lambda_i - i}{\lambda_i-1}$.  Assume that
  $E$ is a locally free sheaf on $X$ of constant rank $m + 1$. The natural
  morphism
  \[
    \pi \colon \Res{n}{d} \to \Hilb^{p} \bigl(\PP(E) \!\bigr)
  \]
  sending a residual flag
  $\varnothing \subset X_e \subset X_{e-1} \subset \dotsb \subset X_{1}$
  to the closed subscheme $X_1 \subset \PP(E)$, is an isomorphism if and only
  if one of the two conditions holds:
  \begin{compactenum}[\upshape (1)]
  \item[{\upshape (2)}] $m \geqslant \lambda_1$ and $\lambda_r \geqslant 2$,
  \item[{\upshape (3)}] $\lambda = (1)$ or
    $\lambda = (m^{r-2}, \lambda_{r-1}, 1)$ where $r \geqslant 2$ and
    $m \geqslant \lambda_{r-1} \geqslant 1$.
  \end{compactenum}
  In both cases, the Hilbert scheme $\Hilb^{p} \!\bigl( \PP(E) \!\bigr)$ is
  smooth over $S$.
\end{theorem}

\begin{proof}
  Suppose that $S = \Spec(k)$ for some algebraically closed field $k$.
  Theorem~1.1 in \cite{St1} demonstrates that Conditions~2 and~3 characterize
  when a nontrivial Hilbert scheme $\Hilb^{p}(\PP^{m})$ has a unique
  Borel-fixed point.  Moreover, Lemma~5.6 in \cite{St1} proves that the target
  $\Hilb^{p}(\PP^{m})$ is nonsingular and irreducible in this situation.
  Proposition~\ref{p:rep} and Corollary~\ref{c:dim} show that the source
  $\Res{\kappa}{d}$ is a smooth projective variety.  Since $\pi$ is injective,
  it is enough to certify that the dimensions of the source and target agree.
  Using Corollary~\ref{c:dim} and equation~\eqref{f:dim}, one verifies that
  the dimension of the lexicographic component in $\Hilb^{p}(\PP^{m})$ equals
  the dimension of the parameter space of residual flags if and only if
  Conditions~2 or~3 holds.

  Suppose that $S$ is any locally noetherian scheme. Lemma~\ref{l:loc} implies
  that the Hilbert scheme is smooth.  Hence, the source and target of the
  morphism $\pi$ are smooth.  Since the induced morphism on fibres over any
  geometric point is an isomorphism, the result follows.
\end{proof}
  

\begin{example}
  The conditions in Theorem~\ref{t:resH} cover the well-known cases of
  hypersurfaces and Grassmannians.  Consider an integer partition
  $\lambda = (\lambda_1^r)$ and set
  $p(t) \coloneqq \sum_{i=1}^{r} \binom{t+\lambda_1-i}{\lambda_1-1}$.  When
  $\lambda_1 = m$, each point in $\Hilb^p \bigl( \PP(E) \!\bigr)$ corresponds
  to a hypersurface of degree $r$ in $\PP(E)$; see Lemma~\ref{l:1pair}.  More
  generally, each point in $\Hilb^p \bigl( \PP(E) \!\bigr)$ corresponds to a
  hypersurface of degree $r$ lying some $\lambda_1$\nobreakdash-dimensional
  linear subspace of $\PP(E)$.  In the special case $r = 1$, the Hilbert
  scheme $\Hilb^p \bigl( \PP(E) \!\bigr)$ is the Grassmannian parametrizing
  $(\lambda_1 -1)$-dimensional linear subspaces in $\PP(E)$.
\end{example}

\begin{example}
  For the integer partition $\lambda = (1^2)$, Theorem~\ref{t:resH} shows
  that each point in the Hilbert scheme $\Hilb^{2} \bigl( \PP(E) \!\bigr)$
  correspond to a hypersurface of degree $2$ lying on some line in $\PP(E)$.
  Alternatively, the Hilbert scheme of two points in $\PP(E)$ is also known to
  be the blow-up of the diagonal of the quotient scheme
  $\PP(E) \times_{\!S} \PP(E) / \mathfrak{S}_2$, where the symmetric group
  $\mathfrak{S}_2$ on two elements acts by permuting the factors in the
  product $\PP(E) \times_{\!S} \PP(E)$.
\end{example}

\begin{remark}
  The geometry of residual flags also explains why the morphism $\pi$ cannot
  be surjective when Conditions~2 and~3 fail to holds.  To avoid these
  conditions, we may assume that $\lambda_r = 1$, $r \geqslant 3$, and
  $m > \lambda_{r-2}$.  The smallest scheme $X_e$ in the residual flag is a
  degree $d_e$ hypersurface in a line.  Hence, the map $\pi$ cannot be
  surjective when $d_e \geqslant 3$.  When $d_e \leqslant 2$, there exists a
  line $\Lambda_e$ containing $X_e$.  The defining properties of a residual
  flag require the line $\Lambda_e$ to be contained in the latent plane
  $\Lambda_{e-1}$ which by assumption has dimension less than $m$.  Since a
  general line is not contained such a plane, the map $\pi$ also not
  surjective in this case.
\end{remark}

\begin{remark}
  The strategy outlined in Question~\ref{q:chow} also leads to a description
  of the Chow ring (and integral cohomology ring) of the Hilbert schemes
  classified in Theorem~\ref{t:resH}.
\end{remark}

\begin{example}
  \label{e:trivial}
  Two trivial Hilbert schemes, not covered by Theorem~\ref{t:resH}, are
  nevertheless particular Grassmannians.  When $\lambda = (m+1)$ and
  $p(t) = \binom{t+m}{m}$, the Hilbert scheme
  $\Hilb^{p} \bigl( \PP(E) \!\bigr)$ is a one point corresponding to closed
  subscheme $\PP(E)$ itself.  When $r = 0$, the Hilbert scheme
  $\Hilb^{0} \bigl( \PP(E) \!\bigr)$ is a one point corresponding to empty
  scheme in $\PP(E)$.
\end{example}

\subsection*{Birationality}
Before examining the birational geometry of the other smooth Hilbert schemes,
we remember that some Hilbert schemes split into a product.  Let
$\lambda = (m^s, \lambda_{s+1}, \lambda_{s+2}, \dotsc, \lambda_r)$ be an
integer partition with $m > \lambda_{s+1}$ and set
$p(t) \coloneqq \sum_{i=1}^{r} \binom{t+\lambda_i-i}{\lambda_i-1}$.  We see
that $p(t) = q_1(t) + q_2(t-s)$ where
$q_1(t) \coloneqq \sum_{i=1}^{s} \binom{t + m - i}{m - 1}$ is the Hilbert
polynomial for a hypersurface of degree $s$ in $\PP(E)$ and
$q_2(t) \coloneqq \sum_{i=1}^{r-s}
\binom{t+\lambda_{i+s}-i}{\lambda_{i+s}-1}$.  Lemma~\ref{l:Cart} shows that
the Hilbert scheme parametrizing of these hypersurfaces is
$\PP \bigl( \Sym^{s} (E^*) \!\bigr)$ and Remark~2 in \cite{Fogarty}*{p.~514}
yields there is a natural splitting
\begin{equation}
  \label{f:split}
  \Hilb^{p} \bigl( \PP(E) \! \bigr)
  \cong \PP \bigl( \Sym^{s} (E^*) \!\bigr) \times_{\!S} \Hilb^{q_2} \bigl(
  \PP(E) \!  \bigr) \, .
\end{equation}

Given an integer partition
$\lambda \coloneqq (\lambda_1, \lambda_2, \dotsc, \lambda_r)$, the new integer
partition $\lambda \cup (1)$ is defined to be
$(\lambda_1, \lambda_2, \dotsc, \lambda_r, 1)$ and has length $r+1$; see
\cite{Mac15}*{Subsection~1.1}.

\begin{proposition}
  \label{p:birat}
  Let $(n,d)$ be a residual type and
  $\lambda \coloneqq (\lambda_1, \lambda_2, \dotsc, \lambda_r)$ be its integer
  partition. Set
  $p(t) \coloneqq \sum_{i=1}^r \binom{t+\lambda_i-i}{\lambda_i-1}$.  Assume
  that one of the following conditions holds:
  \begin{compactenum}[\upshape (1)]
  \item[{\upshape (4)}]
    $\lambda \cup (1) = (m^{r-s-3}, \lambda_{r-s-2}^{s+2}, 1)$ where
    $r-3 \geqslant s \geqslant 0$, and
    $m-1 \geqslant \lambda_{r-s-2} \geqslant 3$, or
  \item[{\upshape (5)}] $\lambda \cup (1) = ( m^{r-s-5}, 2^{s+4}, 1)$ where
    $r-5 \geqslant s \geqslant 0$.
  \end{compactenum}
  The Hilbert scheme $\Hilb^{p+1} \bigl( \PP(E) \!\bigr)$ is smooth over $S$.
  Moreover, a general point on this Hilbert scheme corresponds to the disjoint
  union of a residual flag of type $(n, d)$ and a point.
\end{proposition}

\begin{proof}
  From the splitting in \eqref{f:split}, it is enough to consider an integer
  partition $\lambda = (\lambda_1^r)$ where $r \geqslant 1$ and
  $\lambda_1 \geqslant 2$, excluding the two integer partitions $(2^2)$ and
  $(2^3)$.  By Lemma~\ref{l:loc}, we many assume that the base scheme is the
  spectrum of an algebraically closed field.  Over a field of characteristic
  zero, the Theorem~A in \cite{Ramkumar} classifies all Hilbert schemes with
  precisely two Borel-fixed points; Theorem~B in \cite{Ramkumar} and
  Theorem~1.1 in \cite{St2} also describe this classification over any
  algebraically closed field.  Conditions~(4) and (5) guarantee that the
  Hilbert scheme $\Hilb^{p+1}(\PP^n)$ has two Borel-fixed points.  By
  computing the dimension of the tangent space at the non-lexicographic
  Borel-fixed point, Theorem~A in \cite{Ramkumar} demonstrates that
  $\Hilb^{p+1}(\PP^n)$ is nonsingular.

  To understand a general point, consider the universal flag $X_1$ of type
  $(\lambda_1,s)$ in $\PP(E)$ and the universal closed subscheme $Z$ having
  length one on $\PP(E) = \Hilb^{1} \!\big( \PP(E) \!\bigr)$. As the structure
  map is proper, their intersection determines a closed subset in the product
  $\Res{\lambda_1}{s} \times_{\!S} \PP(E)$.  Let $U$ denote the open complement.
  There is a morphism $\psi \colon U \to \Hilb^{p+1} \bigl( \PP(E) \!\bigr)$
  induced by sending the pair $(X_1, Z)$ to their disjoint union. The source
  and target of $\psi$ are smooth $S$-schemes and, using
  Corollary~\ref{c:dim} and equation~\eqref{f:dim}, one verifies that they
  have the same relative dimension.  Over each geometric point in $S$, the
  induced morphism on the fibres is an open immersion. We conclude that
  $\psi \colon \Res{\lambda_1}{s} \times_{\!S} \PP(E) \dasharrow \Hilb^{p+1}
  \bigl( \PP(E) \!\bigr)$ is birational map.
\end{proof}

\begin{remark}
  Under the hypothesis of Proposition~\ref{p:birat}, the Hilbert scheme
  $\Hilb^{p+1}({\PP(E)})$ and the product $\Res{m}{d} \times_{\!S} \PP(E)$ are
  birational. However, these schemes are not isomorphic. For instance, the
  existence of two different Borel-fixed points on the Hilbert scheme implies
  that there is more than one way to embedded a point of multiplicity $1$ into
  the lexicographic ideal; see Proposition~\ref{p:irr} and
  Proposition~\ref{p:sing}.
\end{remark} 

\begin{example}
  \label{e:pts}
  The integer partition $\lambda = (1^r)$ is associated the contant Hilbert
  polynomial $r$.  The Hilbert scheme $\Hilb^{r}\bigl( \PP(E) \!\bigr)$ is
  known to be smooth in two cases: Theorem~2.4 in \cite{Fogarty} applies when
  $m = 2$ and Equation~(0.2.1) in \cite{Che} applies when $r \leqslant 3$.
  In either case, this Hilbert scheme is birational to the $r$-fold symmetric
  product
  $\PP(E) \times_{\!S} \PP(E) \times_{\!S} \dotsb \times_{\!S} \PP(E) /
  \mathfrak{S}_r$ where the symmetric group $\mathfrak{S}_r$ on $r$ elements
  acts by permuting the factors in the product
  $\PP(E) \times_{\!S} \PP(E) \times_{\!S} \dotsb \times_{\!S} \PP(E)$.
  
  Using the splitting
  \eqref{f:split}, this analysis extends to the integer
  partition $\lambda = (m^{r-s}, 1^{s})$ where $r \geqslant s \geqslant 0$.
  Set $p(t) \coloneqq \sum_{i=1}^{r} \binom{t + \lambda_i -i}{\lambda_i -1}$.
  Assuming $m = 2$ or $r \leqslant 3$, a general point on
  $\Hilb^{p} \bigl( \PP(E) \! \bigr)$ corresponds to the disjoint union of a
  hypersurface of degree $r-s$ and $s$ isolated points.
\end{example}

\begin{remark}
  \label{r:smth}
  The seven conditions in Theorem~\ref{t:main} imply that the Hilbert scheme
  $\Hilb^{p} \!\bigl(\PP(E) \!\bigr)$ is smooth over $S$: Example~\ref{e:pts}
  handles Conditions~1 and~6, Theorem~\ref{t:resH} handles conditions~2
  and~3, Proposition~\ref{p:birat} handles Conditions~4 and~5, and
  Example~\ref{e:trivial} handles Condition~7.  In particular, we have a
  birational description for all of these smooth Hilbert schemes.
\end{remark}

\section{General classification}

The final section completes our classification of smooth Hilbert schemes.  By
identifying enough singular points on Hilbert schemes, we prove that our list
of smooth Hilbert schemes is exhaustive.

\subsection*{Nearly lexicographic points}
We specify a novel point on a Hilbert scheme by perturbing a lexicographic
ideal.  Geometrically, this nearly lexicographic point corresponds to a
residual flag with an embedded point whose nilpotent elements do not lie in
the smallest linear subspace containing the residual flag.

\begin{lemma}
  \label{l:hil}
  Let $\lambda \coloneqq (\lambda_1, \lambda_2, \dotsc, \lambda_r)$ be an
  integer partition and set
  $p(t) \coloneqq \smash{\sum_{i=1}^{r} \binom{t + \lambda_i -i}{\lambda_i
      -1}}$.
  Fix $m > \lambda_1$ and consider both the lexicographic ideal
  $L(\lambda)$ and the monomial ideal
  \[
    J \coloneqq \bigl\langle x_0^{}, \; x_1^{}, \; \dotsc, \;
    x_{n -\smash{\lambda_1} -2}^{}, \;
    x_{m - \smash{\lambda_1} -1}^{\, \smash{2}}, \;
    x_{m - \smash{\lambda_1}}^{}, \;
    x_{n -\smash{\lambda_1}+1}^{}, \; \dotsc, \; x_{n-1}^{}
    \bigr\rangle
  \]
  in the polynomial ring $R = k[x_0, x_1, \dotsc, x_{m}]$.  The closed
  subscheme in $\PP^{m}$ defined by the homogeneous ideal
  $K \coloneqq L(\lambda) \cap J$ has Hilbert polynomial $p+1$ and corresponds
  to a point on the lexicographic component of the Hilbert scheme
  $\Hilb^{p+1}(\PP^{m})$.
\end{lemma}

\begin{proof}
  By Lemma~\ref{l:nonempty}, the Hilbert polynomial of the closed subscheme
  defined by the 
  ideal $L(\lambda)$ is
  $\smash{p(t) = \sum_{i=1}^{r} \binom{t + \lambda_i -i}{\lambda_i -1}}$.
  Proposition~\ref{p:irr} implies that
  $L(\lambda) + J = \langle x_0, x_1, \dotsc, x_{m - 1} \rangle$.  For all
  integers $j$ greater than $1$, the sets
  $\{x_{m - \lambda_1 - 1} x_{m}^{\smash{j-1}}, x_n ^{\smash{j}} \}$ and
  $\{ x_{m}^{\, \smash{j}} \}$ form bases for the $j$-th homogeneous
  components of $R/J$ and $\smash{R/ \bigl( L(\lambda) + J \bigr)}$
  respectively.  It follows that their Hilbert polynomials are the constants
  $2$ and $1$.  From the short exact sequence of graded $R$-modules
  \[
    \begin{tikzcd}[row sep = -1pt, column sep = 17pt]
      0 \arrow[r]
      & \dfrac{R}{L(\lambda) \cap J} \arrow[r]
      & \dfrac{R}{L(\lambda)} \oplus \dfrac{R}{J} \arrow[r]
      & \dfrac{R}{L(\lambda)+J} \arrow[r]
      & 0 \, ,  
    \end{tikzcd}
  \]
  we deduce that $p+1$ is the Hilbert polynomial of the closed subscheme in
  $\PP^{m}$ corresponding to the monomial ideal $K = L(\lambda) \cap J$.

  Using Proposition~\ref{p:irr}, we also deduce that the saturation
  $\bigl( L(\lambda) \mathbin{:} x_{m - 1}^{\infty} \bigr)$ is equal to the
  saturation $(K \mathbin{:} x_{m-1}^{\infty})$.  Since both $L(\lambda)$
  and $K$ are Borel-fixed ideals, Theorem~6 in \cite{Ree} establishes that the
  point corresponding to the homogeneous ideal $K$ lies on the lexicographic
  component.
\end{proof}

\subsection*{Tangent spaces}
We show that these nearly lexicographic points are singular for a special
class of integer partitions.  The Zariski tangent space at the point in the
Hilbert scheme $\Hilb^{p}(\PP^m)$ corresponding to the closed subscheme $X$ in
$\PP^m$ with ideal sheaf $\mathcal{I}_X$ is naturally isomorphic to
$\Hom_{\PP^m}(\mathcal{I}_X, \mathcal{O}_X) =
\Hom_{X}^{}(\mathcal{I}_X^{}/\mathcal{I}_X^2, \mathcal{O}_X^{})$; see
\cite{Kollar}*{Theorem~2.8}.

\begin{lemma}
  \label{l:sing}
  Let $\lambda \coloneqq \bigl(\! (m-1)^{r-s-1}, (m - n)^{s+1} \bigr)$ be an
  integer partition where $r-2 \geqslant s \geqslant 0$ and
  $m - 1 \geqslant n \geqslant 2$.  Set
  $\smash{p(t) \coloneqq \sum_{i=1}^{r} \binom{t + \lambda_i-i}%
    {\lambda_i -1}}$. The Hilbert scheme $\Hilb^{p+1}(\PP^{m})$ is singular at
  the point corresponding to the saturated monomial ideal
  \begin{align*}
    K
    &\coloneqq x_0^{} \cdot \bigl\langle x_0^{}, x_1^{}, \dotsc, x_{m - 1}^{}
      \bigr\rangle + x_{1}^{\, \smash{r-s-1}} \cdot \bigl\langle x_{1}^{},
      x_{2}^{}, \dotsc, x_{\smash{n-1}}^{}, x_{\smash{n}}^{\,\smash{s+1}}
      \bigr\rangle
      = L(\lambda) \cap
      \bigl\langle x_{0}^{\, \smash{2}}, \; x_{1}^{}, x_{2}^{}, \dotsc, x_{m-1}^{}
      \bigr\rangle
  \end{align*}
  in the polynomial ring $R = k[x_0, x_1, \dotsc, x_{m}]$.
\end{lemma}

\begin{proof}
  From the monomial generators for the lexicographic ideal $L(\lambda)$
  appearing in Proposition~\ref{p:irr}, we see that the given monomials
  generate the ideal $K$.  It remains to show that the dimension of the
  Zariski tangent space at the nearly lexicographic point is larger than the
  dimension of the Zariski tangent space at the lexicographic point.
  Equation~\eqref{f:dim} establishes that the dimension of later is less than
  or equal to (with equality holding when $s > 0$)
  \begin{equation}
    \label{f:3}
    \binom{m + r - s - 2}{r - s - 1}
    + \binom{m - n + s + 1}{s + 1}
    +(m - n + 1)(n - 1)+ 2 m - 2  \, . 
  \end{equation}

  To estimate the dimension of the Zariski tangent space at the nearly
  lexicographic point, we examine the sheaf on $\PP^m$ corresponding to the
  graded $R$-module $\Hom_R(K,R/K)$.  Since the variable $x_m$ does not divide
  any of the generators of the ideal $K$, the dimension of this tangent space
  is greater than or equal to $\dim_k \Hom_R(K,R/K)_0$; see
  \cite{RS}*{Lemma~3.1}. Because $K$ is a stable monomial ideal, the
  Eliahou--Kervaire resolution~\cite{PeS}*{Theorem~2.3} yields a homogeneous
  free presentation.  The minimal syzygies among the generators of the ideal
  $K$ are given by the block matrix $\Theta \coloneqq
  \smash{\begin{bmatrix*}
      \textbf{A}_1 & \textbf{A}_2 & \!\!\!\dotsb\!\!\! & \textbf{A}_{m-1} &
      \textbf{B}_1 & \textbf{B}_2 & \!\!\!\dotsb\!\!\! &
      \textbf{B}_{n-1} & \textbf{C} \\
    \end{bmatrix*}}$ where%
  {
    \begin{align*}
      \textbf{A}_i^{\!\!\smash{\textsf{T}}}
      &\coloneqq
        \begin{blockarray}{*{14}{c} c}
          \begin{block}{*{14}{>{$\tiny}c<{$}} c}
            \textcolor{gray}{$x_0^2 \!\!\!\!$} &
            \textcolor{gray}{$x_0^{} x_{1}^{} \!\!\!\!$} &
            \textcolor{gray}{$\dotsb\!\!\!\!$} &
            \textcolor{gray}{$x_0^{} x_{i-2}^{} \!\!\!\!\!\!\!\!$} &
            \textcolor{gray}{$x_0^{}x_{i-1}^{} \!\!\!\!\!\!\!\!$} &
            \textcolor{gray}{$x_0^{}x_{i}^{} \!\!\!\!\!\!\!\!$} &
            \textcolor{gray}{$x_0^{}x_{i+1}^{} \!\!\!\!$} &
            \textcolor{gray}{$\dotsb \!\!\!\!$} &
            \textcolor{gray}{$x_0^{}x_{m -1}^{} \!\!\!\!$} &
            \textcolor{gray}{$x_{1}^{r-s} \!\!\!\!$} &
            \textcolor{gray}{$x_{1}^{r-s-1}x_{1}^{} \!\!\!\!$} &
            \textcolor{gray}{$\dotsb \!\!\!\!$} &
            \textcolor{gray}{$x_{1}^{r-s-1}x_{n-1} \!\!\!\!$} &                        
            \textcolor{gray}{$x_{1}^{r-s-1}x_{n}^{s+1}\!\!\!$} &            
            \\[-1pt]
          \end{block}
          \begin{block}{[*{14}{c}]>{$\tiny}c<{$}} \\[-10pt]
            \;\; 0 \!\!\!\! & 0 \!\!\!\!&
            \dotsb\!\!\!\! & 0 \!\!\!\!\!\!\!\! & -x_i
            \!\!\!\!\!\!\!\!  & x_{i-1} \!\!\!\!\!\!\!\! & 0 \!\!\!\!&
            \dotsb \!\!\!\! & 0 \!\!\!\!& 0 \!\!\!\! &
            0 \!\!\!\! & \dotsb \!\!\!\! & 0 \!\!\!\!&
            0 & \textcolor{gray}{$i$} \\[-2pt]
            \;\; 0 \!\!\!\! & 0 \!\!\!\! & \dotsb
            \!\!\!\! & 0 \!\!\!\!\!\!\!\! & -x_{i+1} \!\!\!\! \!\!\!\!
            & 0 \!\!\!\!\!\!\!\! & x_{i-1} \!\!\!\! & \dotsb
            \!\!\!\! & 0 \!\!\!\! & 0 \!\!\!\! & 0 \!\!\!\! &
            \dotsb \!\!\!\! &
            0 \!\!\!\! & 0 & \textcolor{gray}{$i\!+\!1$} \\[-4pt]
            \;\; \vdots \!\!\!\! & \vdots \!\!\!\! & \ddots
            \!\!\!\! & \vdots \!\!\!\!\!\!\!\! & \vdots \!\!\!\!
            \!\!\!\!\!\!\!\! & \vdots \!\!\!\!\!\!\!\! & \vdots \!\!\!\!&
            \ddots \!\!\!\! & \vdots \!\!\!\! & \vdots
            \!\!\!\! & \vdots \!\!\!\! & \ddots \!\!\!\! &
            \vdots \!\!\!\! & \vdots & \textcolor{gray}{$\vdots$} \\[-2pt]
            \;\; 0 \!\!\!\! & 0 \!\!\!\! & \dotsb
            \!\!\!\! & 0 \!\!\!\!\!\!\!\! & -x_{m-1} \!\!\!\!\!\!\!\!
            & 0 \!\!\!\! \!\!\!\! & 0 \!\!\!\! & \dotsb
            \!\!\!\! & x_{i-1} \!\!\!\! & 0 \!\!\!\! & 0 \!\!\!\! &
            \dotsb\!\!\!\!
            & 0 \!\!\!\! & 0 & \textcolor{gray}{$m\!-\!1$} \\[3pt]
          \end{block} 
        \end{blockarray} \\[-10pt]
     \textbf{B}_{\!j}^{\!\!\smash{\textsf{T}}}
      &\coloneqq
        \begin{blockarray}{*{13}{c} c}
          \begin{block}{*{13}{>{$\tiny}c<{$}} c}
            \textcolor{gray}{$\;\;x_0^2 \!\!\!\!$} &
            \textcolor{gray}{$x_0^{} x_{1}^{} \!\!\!\!$} &
            \textcolor{gray}{$\dotsb \!\!\!\!$} &
            \textcolor{gray}{$x_0^{} x_{m-1}^{} \!\!\!\!$} &
            \textcolor{gray}{$x_1^{r-s} \!\!\!\!$} &
            \textcolor{gray}{$\dotsb \!\!\!\!$} &
            \textcolor{gray}{$x_{1}^{r-s-1} x_{\!j-2}^{} \!\!\!\!$} &
            \textcolor{gray}{$x_{1}^{r-s-1} x_{\!j-1}^{} \!\!\!\!$} &
            \textcolor{gray}{$x_{1}^{r-s-1} x_{\!j}^{} \!\!\!\!$} &
            \textcolor{gray}{$x_{1}^{r-s-1} x_{\!j+1}^{} \!\!\!\!$} &
            \textcolor{gray}{$\dotsb \!\!\!\!$} &
            \textcolor{gray}{$x_{1}^{r-s-1}x_{n-1}^{} \!\!\!\!$} &
             \textcolor{gray}{$x_{1}^{r-s-1}x_{n}^{s+1}\!\!\!$} &
            \\[-1pt]
          \end{block}
          \begin{block}{[*{13}{c}]>{$\tiny}c<{$}} \\[-10pt]
            \;\; 0 \!\!\!\! & 0 \!\!\!\! & \dotsb
            \!\!\!\! & 0 \!\!\!\! & 0 \!\!\!\! &
            \dotsb \!\!\!\! & 0 \!\!\!\! & -x_{\!j}
            \!\!\!\! & x_{\!j-1} \!\!\!\! & 0 \!\!\!\! &
            \dotsb \!\!\!\! & 0 \!\!\!\! & 0 &
            \textcolor{gray}{$j$} \\[-2pt]
            \;\; 0 \!\!\!\! & 0 \!\!\!\! & \dotsb
            \!\!\!\! & 0 \!\!\!\! & 0 \!\!\!\! &
            \dotsb \!\!\!\! & 0 \!\!\!\! & -x_{\!j+1}
            \!\!\!\! & 0 \!\!\!\! & x_{\!j-1} \!\!\!\! &
            \dotsb \!\!\!\! & 0 \!\!\!\! & 0 &
            \textcolor{gray}{$j\!+\!1$} \\[-4pt]
            \;\; \vdots \!\!\!\! & \vdots \!\!\!\! & \ddots
            \!\!\!\! & \vdots \!\!\!\! & \vdots \!\!\!\! &
            \ddots\!\!\!\! & \vdots \!\!\!\! & \vdots
            \!\!\!\! & \vdots \!\!\!\! & \vdots \!\!\!\! &
            \ddots \!\!\!\! & \vdots \!\!\!\! & \vdots &
            \textcolor{gray}{$\vdots$} \\[-2pt]
            \;\; 0 \!\!\!\! & 0 \!\!\!\! & \dotsb
            \!\!\!\! & 0 \!\!\!\! & 0 \!\!\!\! &
            \dotsb \!\!\!\! & 0 \!\!\!\! & -x_{n-1}
            \!\!\!\! & 0 \!\!\!\! & 0 \!\!\!\! &
            \dotsb \!\!\!\! & x_{\!j-1} \!\!\!\! & 0 &
            \textcolor{gray}{$n\!-\!1$} \\[3pt]
          \end{block} 
        \end{blockarray} \\[-10pt]
     \textbf{C}^{\!\smash{\textsf{T}}}
      &\coloneqq
        \begin{blockarray}{*{13}{c} c}
          \begin{block}{*{13}{>{$\tiny}c<{$}} c}
            \textcolor{gray}{$\;\;x_0^2 \!\!\!\!$} &
            \textcolor{gray}{$x_0^{} x_{1}^{} \!\!\!\!$} &
            \textcolor{gray}{$\!\!\!\!\dotsb \!\!\!\!\!\!\!\!$} &
            \textcolor{gray}{$x_0^{} x_{n-1}^{} \!\!\!\!$} &            
            \textcolor{gray}{$x_{0}^{} x_{n}^{} \!\!\!\!$} &
            \textcolor{gray}{$x_0^{} x_{n+1}^{} \!\!\!\!$} &                
            \textcolor{gray}{$\!\!\!\!\dotsb \!\!\!\!\!\!\!\!$} &
            \textcolor{gray}{$x_0^{} x_{m-1}^{} \!\!\!\!$} &              
            \textcolor{gray}{$x_{1}^{r-s} \!\!\!\!$} &
            \textcolor{gray}{$x_{1}^{r-s-1} x_{2}^{} \!\!\!\!$} &
            \textcolor{gray}{$\!\!\!\!\dotsb \!\!\!\!\!\!\!\!$} &            
            \textcolor{gray}{$x_{1}^{r-s-1} x_{n-1}^{} \!\!\!\!$} &
            \textcolor{gray}{$\!\!x_{1}^{r-s-1}x_{n}^{s+1}\!\!\!$} &
            \\[-1pt]
          \end{block}
          \begin{block}{[*{13}{c}]>{$\tiny}c<{$}} \\[-10pt]
            \;\; 0 \!\!\!\!& 0 \!\!\!\!& \!\!\!\!\dotsb
            \!\!\!\!\!\!\!\! & 0 \!\!\!\!& \!\!\!\! -x_{1}^{r-x-1} 
            x_{n}^{s} \!\!\!\! \!\!\!\! & 0 \!\!\!\! &
            \!\!\!\!\dotsb\!\!\!\! \!\!\!\! & 0 \!\!\!\! & 0 \!\!\!\!
            & 0 \!\!\!\! & \!\!\!\!\dotsb\!\!\!\! \!\!\!\! & 0
            \!\!\!\! & x_{0}^{} & \textcolor{gray}{$0$}
            \\[-2pt]
            \;\; 0 \!\!\!\!& 0 \!\!\!\!& \!\!\!\!\dotsb
            \!\!\!\!\!\!\!\! & 0 \!\!\!\! & 0 \!\!\!\! & 0 \!\!\!\! &
            \!\!\!\!\dotsb \!\!\!\!\!\!\!\! & 0 \!\!\!\! & \!\!\!\!
            -x_{n}^{s+1} \!\!\!\! \!\!\!\! & 0 \!\!\!\! &
            \!\!\!\!\dotsb \!\!\!\!\!\!\!\! & 0 \!\!\!\!& x_{1} &
            \textcolor{gray}{$1$} \\[-2pt]
            \;\; 0 \!\!\!\! & 0 \!\!\!\! & \!\!\!\!\dotsb
            \!\!\!\!\!\!\!\! & 0 \!\!\!\! & 0 \!\!\!\! & 0 \!\!\!\! &
            \!\!\!\!\dotsb \!\!\!\!\!\!\!\! & 0 \!\!\!\! & 0 \!\!\!\!
            & \!\!\!\! -x_{n}^{s+1} \!\!\!\! \!\!\!\! & \!\!\!\!\dotsb
            \!\!\!\!\!\!\!\! & 0 \!\!\!\! & x_{2} &
            \textcolor{gray}{$2$} \\[-4pt]
            \;\; \vdots \!\!\!\! & \vdots \!\!\!\! &
            \!\!\!\!\ddots\!\!\!\!\!\!\!\! & \vdots \!\!\!\! & \vdots
            \!\!\!\! & \vdots \!\!\!\! &
            \!\!\!\!\ddots \!\!\!\!\!\!\!\! & \vdots \!\!\!\! & \vdots
            \!\!\!\! & \vdots \!\!\!\! & \!\!\!\!\ddots
            \!\!\!\!\!\!\!\!  & \vdots \!\!\!\! & \vdots &
            \textcolor{gray}{$\vdots$} \\[-2pt]
            \;\; 0 \!\!\!\! & 0 \!\!\!\! & \!\!\!\!\dotsb
            \!\!\!\!\!\!\!\! & 0 \!\!\!\! & 0 \!\!\!\! & 0 \!\!\!\! &
            \!\!\!\!\dotsb \!\!\!\!\!\!\!\! & 0 \!\!\!\! & 0 \!\!\!\!
            & 0 \!\!\!\! & \!\!\!\!\dotsb \!\!\!\!\!\!\!\! & \!\!\!\!
            -x_{n}^{s+1} \!\!\!\! \!\!\!\!  & x_{n-1} &
            \textcolor{gray}{$m\!-\!1$} \\[3pt]
          \end{block} 
        \end{blockarray} \, .
      \\[-30pt]      
    \end{align*}}%
  for all $1 \leqslant i \leqslant m-1$ and all $1 \leqslant j \leqslant n-1$.
  It follows that
  $\Hom_R(K,R/K) = \operatorname{Ker} \bigl( \Hom_R(\Theta, R/K) \!\bigr)$.

  The two $(m+n) \times 1$-matrices defined by%
  {
    \begin{align*}
      \begin{blockarray}{r *{10}{c} c}
        \begin{block}{r *{10}{>{$\tiny}c<{$}} c}
          & \textcolor{gray}{$\;\; x_0^2$} &
          \textcolor{gray}{$x_0^{}x_{1}^{}$} &
          \textcolor{gray}{$x_0^{}x_{1}^{}$} &
          \textcolor{gray}{$\dotsb$} &
          \textcolor{gray}{$x_0^{}x_{m-1}^{}$} &
          \textcolor{gray}{$x_1^{r-s-1}$} &
          \textcolor{gray}{$x_1^{r-s-1}x_{2}^{}$} &
          \textcolor{gray}{$\dotsb$} &
          \textcolor{gray}{$x_1^{r-s-1}x_{n-1}^{}$} &
          \textcolor{gray}{$x_1^{r-s-1}x_{n}^{s+1}$} &          
          \\[-1pt]
        \end{block}      
        \begin{block}{r [*{10}{c}]>{$\tiny}c<{$}} \\[-13pt]
          \textbf{D}_{0}^{\!\smash{\textsf{T}}} \coloneqq & \;\; x_0 & x_1 &
          x_2 & \dotsb & x_n & 0 & 0 & \dotsb & 0 & 0 &
          \textcolor{gray}{$0$} \\[1pt]         
        \end{block} \\[-7pt]
        \begin{block}{r *{10}{>{$\tiny}c<{$}} c}
          & \textcolor{gray}{$\;\; x_0^2$} &
          \textcolor{gray}{$x_0^{}x_{1}^{}$} &
          \textcolor{gray}{$x_0^{}x_{1}^{}$} &
          \textcolor{gray}{$\dotsb$} &
          \textcolor{gray}{$x_0^{}x_{m-1}^{}$} &
          \textcolor{gray}{$x_1^{r-s-1}$} &
          \textcolor{gray}{$x_1^{r-s-1}x_{2}^{}$} &
          \textcolor{gray}{$\dotsb$} &
          \textcolor{gray}{$x_1^{r-s-1}x_{n-1}^{}$} &
          \textcolor{gray}{$x_1^{r-s-1}x_{n}^{s+1}$} &          
          \\[-1pt]
        \end{block}     
        \begin{block}{r [*{10}{c}]>{$\tiny}c<{$}} \\[-13pt]
          \textbf{D}_{1}^{\!\smash{\textsf{T}}} \coloneqq & \;\; 0 & 0 & 0 &
          \dotsb & 0 & x_1^{} & x_2^{} & \dotsb & x_{n-1}^{} & x_{n}^{s+1} &
          \textcolor{gray}{$1$} \\[1pt]
        \end{block}         
      \end{blockarray}  \\[-30pt]
    \end{align*}}%
  satisfy $\Theta^{\!\smash{\textsf{T}}} \, \mathbf{D}_{0} = \mathbf{0}$ and
  $\Theta^{\!\smash{\textsf{T}}} \, \mathbf{D}_{1} = \mathbf{0}$.  Thus, for
  all $2 \leqslant i \leqslant m$, the column in the product
  $x_{i} \, \mathbf{D}_0$ represents a nonzero element in
  $\Hom_R(K, R/K)_0$.  The column in the product of the matrix $\mathbf{D}_1$
  with any monomial of degree $r-s-1$ in the variables $x_1, x_2, \dotsc, x_m$
  (excluding $x_1^{r-s-1}$) also represents a nonzero element in
  $\Hom_R(K, R/K)_0$.  There are $(m-1) + \binom{m+r-s-2}{r-s-1}-1$ columns of
  these products.  Since all entries in the product $x_0 \, \Theta$ lie in the
  ideal $K$, the $m + n$ columns of the square matrix%
  {
    \begin{align*}
      \begin{blockarray}{*{9}{c} c}
        \begin{block}{*{9}{>{$\tiny}c<{$}} c}
          \textcolor{gray}{$1$} &
          \textcolor{gray}{$2$} &
          \textcolor{gray}{$\!\!\!\!\dotsb\!\!\!\!$} &
          \textcolor{gray}{$m$} &            
          \textcolor{gray}{$1$} &
          \textcolor{gray}{$2$} &                
          \textcolor{gray}{$\!\!\!\!\dotsc\!\!\!\!$} &
          \textcolor{gray}{$n-1$} &              
          \textcolor{gray}{$n$} &
          \\[-1pt]
        \end{block}
        \begin{block}{[*{9}{c}]>{$\tiny}c<{$}} \\[-10pt]
          x_0 x_m &  0 & \!\!\!\!\dotsb\!\!\!\! & 0 & 0 & 0 &
          \!\!\!\!\dotsb\!\!\!\! & 0 & 0 & \textcolor{gray}{$x_0^2$} \\[-2pt]
          0 &  x_0 x_m & \!\!\!\!\dotsb\!\!\!\! & 0 & 0 & 0 & \!\!\!\!\dotsb
          \!\!\!\! & 0 & 0 & \textcolor{gray}{$x_0x_1$} \\[-4pt]
          \vdots & \vdots & \!\!\!\!\ddots\!\!\!\! & \vdots & \vdots & \vdots
          & \!\!\!\!\ddots \!\!\!\! & \vdots & \vdots &
          \textcolor{gray}{$\vdots$} \\[-2pt]
          0 & 0 & \!\!\!\!\dotsb\!\!\!\! & x_0 x_m & 0 & 0 & \!\!\!\!\dotsb
          \!\!\!\! & 0 & 0 & \textcolor{gray}{$x_0 x_{m-1}$} \\[-2pt]
          0 & 0 & \!\!\!\!\dotsb \!\!\!\! & 0 & x_0^{} x_n^{r-s-1} \!\!\!\! &
          0 & \!\!\!\!\dotsb\!\!\!\! & 0 & 0 & \textcolor{gray}{$x_1^{r-s}$}
          \\[-2pt]
          0 & 0 & \!\!\!\!\dotsb \!\!\!\! & 0 & 0 & x_0^{} x_n^{r-s-1}
          \!\!\!\! & \!\!\!\!\dotsb\!\!\!\! & 0 & 0
          & \textcolor{gray}{$x_1^{r-s-1} x_2^{}$} \\[-4pt]
          \vdots & \vdots & \!\!\!\!\ddots\!\!\!\! & \vdots & \vdots & \vdots
          & \!\!\!\!\ddots \!\!\!\! & \vdots & \vdots &
          \textcolor{gray}{$\vdots$} \\[-2pt]
          0 &  0 & \!\!\!\!\dotsb \!\!\!\! & 0 & 0 & 0 & \!\!\!\!\dotsb
          \!\!\!\! & x_0^{} x_n^{r-s-1} \!\!\!\! & 0 & \textcolor{gray}{$x_1^{r-s-1}
            x_{n-1}^{}$} \\[-2pt]
          0 & 0 & \!\!\!\!\dotsb \!\!\!\! & 0 & 0 & 0 & \!\!\!\!\dotsb
          \!\!\!\! & 0 & x_0^{} x_m^{r-1} &
          \textcolor{gray}{$x_1^{r-s-1}x_{n}^{s+1}$} \\[3pt]
         \end{block} 
       \end{blockarray}
      \\[-30pt]    
    \end{align*}}%
  represent nonzero elements in $\Hom_R(K, R/K)_0$.  Similarly, all of entries
  in the bottom $n$ rows of the matrix $\Theta$ when multiplied by the
  monomial $x_1^{r-s-1}$ lie in the ideal $K$.  Hence, for all
  $n \leqslant j \leqslant m$, the columns of the matrices%
  {
    \begin{align*}
      &
      \begin{blockarray}{*{4}{c} c}
        \begin{block}{*{4}{>{$\tiny}c<{$}} c}
          \textcolor{gray}{$1 \!\!\!\!$} &
          \textcolor{gray}{$2 \!\!\!\!$} &                
          \textcolor{gray}{$\!\!\dotsb\!\!\!\!\!\!$} &
          \textcolor{gray}{$n-1$} &              
          \\[-1pt]
        \end{block}
        \begin{block}{[*{4}{c}]>{$\tiny}c<{$}} \\[-10pt]
          0 \!\!\!\! & 0 \!\!\!\!& \!\!\dotsb \!\!\!\!\!\! & 0 &
          \textcolor{gray}{$x_0^2$} \\[-2pt]
          0 \!\!\!\! & 0 \!\!\!\! & \!\!\dotsb \!\!\!\!\!\! & 0 &
          \textcolor{gray}{$x_0x_1$} \\[-4pt]
          \vdots \!\!\!\! & \vdots \!\!\!\! & \!\!\ddots \!\!\!\!\!\! & \vdots
          & \textcolor{gray}{$\vdots$} \\[-2pt]
          0 \!\!\!\! & 0 \!\!\!\! & \!\!\dotsb \!\!\!\!\!\! & 0
          & \textcolor{gray}{$x_0 x_{m-1}$} \\[-2pt]
          x_1^{r-s-1} x_j^{} \!\!\!\! & 0 \!\!\!\! & \!\!\dotsb\!\!\!\!\!\! & 0 
          & \textcolor{gray}{$x_1^{r-s}$} \\[-2pt]
          0 \!\!\!\! & x_1^{r-s-1} x_j^{} \!\!\!\! & \!\!\dotsb
          \!\!\!\!\!\! & 0 & \textcolor{gray}{$x_1^{r-s-1} x_2^{}$} \\[-4pt]
          \vdots \!\!\!\! & \vdots \!\!\!\! & \!\!\ddots \!\!\!\!\!\!
          & \vdots  & \textcolor{gray}{$\vdots$} \\[-2pt]
          0 \!\!\!\! & 0 \!\!\!\! & \!\!\dotsb \!\!\!\!\!\! & x_1^{r-s-1} x_j^{}
          & \textcolor{gray}{$x_1^{r-s-1} x_{n-1}^{}$} \\[-2pt]
          0 \!\!\!\! & 0 \!\!\!\! & \!\!\dotsb \!\!\!\!\!\! & 0 &
          \textcolor{gray}{$x_1^{r-s-1}x_{n}^{s+1}$} \\[4pt]
        \end{block} 
      \end{blockarray} 
      \!\!\!\! , \!\!
      &\hspace{2pt}&
      \begin{blockarray}{*{4}{c} c}
        \begin{block}{*{4}{>{$\tiny}c<{$}} c}
          \textcolor{gray}{$1$} &
          \textcolor{gray}{$2$} &                
          \textcolor{gray}{$\dotsb$} &
          \textcolor{gray}{$\!\!\!\!\!\!\binom{m-n+s+1}{s+1}\!-\!1$} &              
          \\[-1pt]
        \end{block}
        \begin{block}{[*{4}{c}]>{$\tiny}c<{$}} \\[-10pt]
          0 & 0 & \dotsb & 0 & \textcolor{gray}{$x_0^2$}
          \\[-2pt]
          0 & 0 & \dotsb & 0 & \textcolor{gray}{$x_0x_1$}
          \\[-4pt]
          \vdots & \vdots & \ddots & \vdots & 
          \textcolor{gray}{$\vdots$} \\[-2pt]
          0 & 0 & \dotsb & 0
          & \textcolor{gray}{$x_0 x_{m-1}$} \\[-2pt]
          0 & 0 & \dotsb & 0 
          & \textcolor{gray}{$x_1^{r-s}$} \\[-2pt]
          0 & 0 & \dotsb & 0
          & \textcolor{gray}{$x_1^{r-s-1} x_2^{}$} \\[-4pt]
          \vdots & \vdots & \ddots & \vdots 
          & \textcolor{gray}{$\vdots$} \\[-2pt]
          0 & 0 & \dotsb & 0
          & \textcolor{gray}{$x_1^{r-s-1} x_{n-1}^{}$} \\[-2pt]
          x_1^{r-s-1} x_{n\vphantom{1}}^{s} x_{n+1}^{} \!\!\!\!
          & x_1^{r-s-1} x_{n\vphantom{1}}^{s} x_{n+2}^{} \!\!\!\!\!\!
          & \dotsb & \!\!\!\!\!\! x_1^{r-s-1} x_{m}^{s+1} &
          \textcolor{gray}{$x_1^{r-s-1}x_{n}^{s+1}$} \\[4pt]
        \end{block} 
      \end{blockarray}
      \hspace{-6pt}
      \\[-30pt]    
    \end{align*}}%
  represent nonzero elements in $\Hom_R(K, R/K)_0$.  Each nonzero entry in
  the bottom row of the second matrix is the product of $x_1^{r-s-1}$ and a
  monomial of degree $s+1$ in the variables $x_n , x_{n+1}, \dotsc, x_m$
  (excluding $x_n^{s+1}$).  Hence, there are $\binom{m-n+s+1}{s+1} -1$ columns
  in this matrix.
  
  Our total number of distinct columns representing nonzero elements in
  $\Hom_R(K, R/K)_0$ is
  \begin{align*}
    N &\coloneqq (m-1) + \binom{m+r-s-2}{r-s-1} -1 + (m+n) + (m-n+1)(n-1) 
      + \binom{m-n+s+1}{s+1} -1 \, . 
  \end{align*}
  By comparing their nonzero monomial entries, we see that these $N$ columns
  are linearly independent. The difference between the number $N$ and
  \eqref{f:3} is $n-1$. As $n \geqslant 2$, we conclude that the Hilbert
  scheme is singular at the point corresponding to the monomial ideal $K$.
\end{proof}

\begin{proposition}
  \label{p:sing}
  Let $\lambda \coloneqq (\lambda_1, \lambda_2, \dotsc, \lambda_r)$ be an
  integer partition such that $\lambda \cup (1)$ has at least three distinct
  parts or $\lambda \coloneqq (\smash{\lambda_1^{r-s-1}},1^{s+1})$ where
  $r-2 \geqslant s \geqslant 0$ and $\lambda_1 > 1$.  Fix $m > \lambda_1$, set
  $\smash{p(t) \coloneqq \sum_{i=1}^{r} \binom{t + \lambda_i -i}{\lambda_i
      -1}}$, and consider the monomial ideal
  \[
    J \coloneqq \bigl\langle x_0^{}, \; x_1^{}, \; \dotsc, \;
    x_{m - \smash{\lambda_1} -2}^{}, \;
    x_{m - \smash{\lambda_1} -1}^{\, \smash{2}}, \;
    x_{m - \smash{\lambda_1}}^{}, \;
    x_{m - \smash{\lambda_1}+1}^{}, \; \dotsc, \;
    x_{m - 1}^{} \bigr\rangle \, . 
  \]
  in the polynomial ring $R = k[x_0,x_1, \dotsc, x_{m}]$. The Hilbert scheme
  $\Hilb^{p+1}(\PP^{m})$ is singular at the point corresponding to the
  saturated monomial ideal $K \coloneqq L(\lambda) \cap J$.
\end{proposition}

\begin{proof}
  Lemma~\ref{l:hil} shows that the nearly lexicographic point lies on the
  lexicographic component of the Hilbert scheme $\Hilb^{p+1}(\PP^{m})$.  We
  reduce the analysis to a special case.

  The inclusion
  $\langle x_0, x_1, \dotsc, x_{m - \smash{\lambda_1}-2} \rangle \subset K$
  implies that the nearly lexicographic point is contained in a
  $(\lambda_1 + 1)$-plane.  Hence, we may assume that $m = \lambda_1 + 1$.

  Assume that $\lambda \cup (1)$ has at least three distinct parts.  The
  lexicographic ideal $L\!\bigl( \lambda \cup (1) \!\bigr)$ is the flat limit
  of a one-parameter family whose general member is the sum of the
  lexicographic ideal $L(\lambda)$ and the ideal of a disjoint point.  Since
  the dimension of the Zariski tangent space at a point in family is an
  upper-semicontinuous function, we may also assume that $\lambda_r > 1$.
  Applying Corollary~\ref{c:lex}, the lexicographic ideal
  $L(\lambda \cup (1))$ determines a residual flag
  $\varnothing \subset X_{e} \subset X_{e-1} \subset \dotsb \subset X_1$ in
  $\PP^{m}$. The hypotheses ensure that $e \geqslant 3$.  The closed subscheme
  $X_{e-2}$ lies in some linear space $\Lambda \subseteq \PP^{m}$.  We may
  deform the scheme $X_{e-2}$ in the linear space $\Lambda$ and leave the rest
  of the residual flag $X_{e-3} \subset X_{e-4} \subset \dotsb \subset X_1$
  unchanged.  If the closed scheme $X_{e-2}$ corresponds to a singular point
  on the Hilbert scheme in $\Lambda$, then it follows that the closed scheme
  $X_1$ corresponds to a singular point on the Hilbert scheme.  Thus, we may
  assume that $e = 3$.

  With these reductions, it suffices to consider the integer partition
  $\lambda = \bigl( (m-1)^{r-s-1}, (m - n)^{s+1} \bigr)$ where
  $r - 2 \geqslant s \geqslant 0$ and $m - 1 \geqslant n \geqslant 2$. In this
  special case, Lemma~\ref{l:sing} proves that the dimension of the Zariski
  tangent space at the nearly lexicographic point exceeds the dimension of the
  lexicographic component.  Therefore, the Hilbert scheme
  $\Hilb^{p+1}(\PP^{m})$ is singular at the point corresponding to the
  saturated monomial ideal $K \coloneqq L(\lambda) \cap J$.
\end{proof}

\subsection*{Other singular examples}
In addition to our family of singular Hilbert schemes, the classification of
smooth Hilbert schemes relies on three other singular families.

\begin{example}
  \label{e:gaps}
  Two familiar Hilbert schemes explain the curious gap between Conditions~4
  and~5 in Proposition~\ref{p:birat}.  By appealing to the splitting in
  \eqref{f:split}, it is enough to understand integer partitions $(2,1)$,
  $(2^2, 1)$ and $(2^3, 1)$. The first of these is already covered by both
  Theorem~\ref{t:resH} and Example~\ref{e:pts}.  In contrast, the Hilbert
  schemes in the other two cases are singular.

  The integer partition $(2^2, 1)$ is associated to the Hilbert polynomial
  $2t+2$. The Hilbert scheme $\Hilb^{2t+2}(\PP^{m})$ is singular; it has two
  irreducible components.  A general point on one component corresponds to a
  pair of skew lines and a general point on the other corresponds to the union
  of a plane conic and an isolated point; compare with
  \cite{CCN}*{Theorem~1.1}.
  
  The integer partition $(2^3, 1)$ is associated to the Hilbert polynomial
  $3t+1$. The Hilbert scheme $\Hilb^{3t+1}(\PP^{m})$ is again singular
  because it has two irreducible components.  A general point in first
  component corresponds to a twisted cubic curve and a general point in the
  other corresponds to the union of a plane cubic and an isolated points;
  compare with \cite{PS}*{Theorem}.
\end{example}

\begin{example}
  \label{e:4pts}
  For completeness, we also add an explicit description of another well-known
  singular Hilbert scheme.  For any nonnegative integer $s$, consider the
  integer partition $\lambda = (1^{s+4})$ whose associated Hilbert polynomial
  is the constant $s+4$.  For all $m \geqslant 3$, the Hilbert scheme
  $\Hilb^{s+4}(\PP^{m})$ is singular at the point corresponding to the
  saturated homogeneous ideal
  \[
    B(s) \coloneqq \bigl\langle x_0^{}, \; x_1^{}, \; \dotsc, \;
    x_{m - 4}^{}, \; x_{m - 3}^{\,\smash{2}}, \;
    x_{m - 3}^{} \, x_{m - 2}^{}, \;
    x_{m - 3}^{} \, x_{m - 1}^{}, \;
    x_{m - 2}^{\,\smash{2}}, \;
    x_{m - 2}^{} \, x_{m - 1}^{}, \;
    x_{m - 1}^{\,\smash{s+2}} \bigr\rangle
  \]
  in the polynomial ring $R = k[x_0, x_1, \dotsc, x_{m}]$; compare with
  \cite{Che}*{Lemma~1.4}.
\end{example}

\begin{theorem}
  Let $E$ be a locally free sheaf on a locally noetherian scheme $S$ of
  constant rank $m+1$ and let $p$ be a numerical polynomial. The Hilbert
  scheme $\Hilb^{p}({\PP(E)})$ is smooth and irreducible over $S$ if and only
  if there exists an integer partition
  $\lambda = (\lambda_1, \lambda_2, \dotsc, \lambda_r)$ such that
  $p(t) = \sum_{i=1}^{r} \binom{t + \lambda_i - i}{\lambda_i -1}$ and one of
  the following seven conditions holds:
  \begin{compactenum}[\upshape (1)]
  \item $m = 2 \geqslant \lambda_1$
  \item $m \geqslant \lambda_1$ and $\lambda_r \geqslant 2$, 
  \item $\lambda = (1)$ or $\lambda = (m^{r-2}, \lambda_{r-1}, 1)$ where
    $r \geqslant 2$ and $m \geqslant \lambda_{r-1} \geqslant 1$,
  \item $\lambda = (m^{r-s-3}, \lambda_{r-s-2}^{s+2}, 1)$ where
    $r - 3 \geqslant s \geqslant 0$ and
    $m - 1 \geqslant \lambda_{r-s-2} \geqslant 3$,
  \item $\lambda = (m^{r-s-5}, 2^{s+4}, 1)$ where
    $r-5 \geqslant s \geqslant 0$,
  \item $\lambda = (m^{r-3}, 1^3)$ where $r \geqslant 3$,
  \item $\lambda = (m+1)$ or $r = 0$.
  \end{compactenum}  
\end{theorem}

\begin{proof}
  Remark~\ref{r:smth} already shows that each condition implies that the
  Hilbert scheme is smooth.  Hence, it suffices to prove that the Hilbert
  scheme has a singular point when the integer partition
  $\lambda \coloneqq (\lambda_1, \lambda_2, \dotsc, \lambda_r)$ does not
  satisfy Conditions~1--7.  To bypass Conditions~1 and~2, we must have
  $m \geqslant 3$ and $\lambda_r = 1$.  By Lemma~\ref{l:loc}, we may assume
  that $S$ is the spectrum of an algebraically closed field.  Using the
  splitting in \eqref{f:split}, we may also assume that $m > \lambda_1$.  For
  the remaining integer partitions with one distinct part,
  Example~\ref{e:4pts} describes a singularity.  When the integer partition
  has two distinct parts, there are three outstanding cases, namely
  $\lambda = (2^2,1)$, $\lambda = (2^3,1)$, or
  $\lambda = (\smash{\lambda_1^{r-s-2}}, 1^{s+2})$ where
  $r-1 \geqslant s \geqslant 0$ and $\lambda_1 > 1$.  Example~\ref{e:gaps} and
  Proposition~\ref{p:sing} exhibit their singularities.  Finally,
  Proposition~\ref{p:sing} also identifies a singular point whenever the
  integer partition has at least three distinct parts.
\end{proof}

\subsection*{Acknowledgements}
We thank Dave Anderson, Sam Payne, Mike Roth, and Michael~E.~Stillman for
their suggestions.  Computational experiments done in
\emph{Macaulay2}~\cite{M2} were indispensable.  GGS was partially supported by
the Natural Sciences and Engineering Research Council of Canada (NSERC) and
the Knut and Alice Wallenberg Foundation.


\raggedright

\end{document}